\documentclass[11pt]{article}

\usepackage{fullpage}
\usepackage{amsmath,amssymb,amsthm}
\usepackage{cite}
\usepackage[usenames]{color}
\definecolor{plum}  {rgb}{.4,0,.4}
\definecolor{BrickRed} {rgb}{0.6,0,0}
\usepackage[plainpages=false,pdfpagelabels,colorlinks=true,linkcolor=BrickRed,citecolor=plum]{hyperref}
\usepackage{fourier} % math & rm
\usepackage[scaled=0.875]{helvet} % ss
\newlength{\minipagewidth}
\newcommand{\bookbox}[1]{
\par\medskip\noindent
\framebox[\textwidth]{
\begin{minipage}{\minipagewidth}
{#1}
\end{minipage} } \par\medskip }
\def\cB{{\mathcal B}}

\def\cF{{\mathcal F}}

\def\cP{{\mathcal P}}

\def\sX{{\mathsf X}}

\def\Reals{\mathbb{R}}
\def\Naturals{\mathbb{N}}
\def\E{\mathbb{E}}
\def\Prob{{\mathbb{P}}}
\def\1{\mathbf{1}}
\def\argmin{\operatornamewithlimits{arg\,min}}
\def\deq{\triangleq}
\def\ave#1{\langle #1 \rangle}
\def\Ave#1{\left\langle #1 \right\rangle}

\def\Ric{\operatorname{Ric}}
\def\d{{\mathrm d}}
\def\tZ{{\tilde{Z}}}
\def\edge#1#2{\{ #1,#2 \}}

\newtheorem{theorem}{Theorem}
\newtheorem{lemma}{Lemma}
\newtheorem{proposition}{Proposition}
\newtheorem{corollary}{Corollary}

\newtheorem{remark}{Remark}

\def\an#1{{#1}}
\def\anr#1{{#1}}
\def\mr#1{{#1}}

\begin{document}

\title{\bf Online Discrete Optimization in Social Networks\\
in the Presence of Knightian Uncertainty\thanks{Research supported in part by the National Science Foundation under grant no.~CCF-1017564, by the Office of Naval Research under grant no.~N00014-12-1-0998, and by the UIUC College of Engineering under Strategic Research Initiative on Cognitive and Algorithmic Decision-Making.}}

\author{Maxim Raginsky\thanks{Department of Electrical and Computer Engineering and the Coordinated Science Laboratory, University of Illinois at Urbana--Champaign, Urbana, IL 61801, USA. E-mail: maxim@illinois.edu.}
\and Angelia Nedi\'c\thanks{Department of Industrial and Enterprise Systems Engineering and the Coordinated Science Laboratory, University of Illinois at Urbana--Champaign, Urbana, IL 61801, USA. E-mail: angelia@illinois.edu.}}

\date{First version: July 1, 2013;\,\, revision: January 27, 2015}

\maketitle
\begin{abstract}
We study a model of collective real-time decision-making (or learning) in a social network operating in an uncertain environment, for which no a priori probabilistic model is available. Instead, the environment's impact on the agents in the network is seen through a sequence of cost functions, revealed to the agents in a causal manner only after all the relevant actions are taken. There are two kinds of costs: individual costs incurred by each agent and local-interaction costs incurred by each agent and its neighbors in the social network. Moreover, agents have inertia: each agent has a default mixed strategy that stays fixed regardless of the state of the environment, and must expend effort to deviate from this strategy in order to respond to cost signals coming from the environment. We construct a decentralized strategy, wherein each agent selects its action based only on the costs directly affecting it and on the decisions made by its neighbors in the network. In this setting, we quantify social learning in terms of regret, which is given by the difference between the realized network performance over a given time horizon and the best performance that could have been achieved in hindsight by a fictitious centralized entity with full knowledge of the environment's evolution. We show that our strategy achieves the regret that scales polylogarithmically with the time horizon and polynomially with the number of agents and the maximum number of neighbors of any agent in the social network.
\end{abstract}

\thispagestyle{empty}

\section{Introduction}

\subsection{Risk vs.\ uncertainty in social learning and optimization}

Decision-making and optimization based on information dispersed among a large number of agents are topics of significant current interest, from both theoretical and practical points of view. Existing literature, which is vast, covers a wide variety of models with different assumptions on the information structure, i.e., who is allowed to observe what, and on the agents' capabilities, i.e., what they are allowed to or able to do with their observations. For example, canonical models of Bayesian learning \cite{Chamley_book} assume complete and truthful sharing of all relevant information among all agents, who are also endowed with essentially unlimited computational power. There are also generalizations to noisy signals, but it is typically assumed that the source of noise is nonadversarial. A related framework of Bayesian dynamic games \cite{Harsanyi1,KalaiLehrer} considers sequential decisions by a large collection of agents, where each agent has perfect recall of all decisions made in the past (but not necessarily of all {\em information} used to arrive at those decisions).

Recently, however, emphasis has shifted towards decision-making in {\em social networks}, where information sharing is limited to small groups of agents --- e.g., when an individual is deciding whether to buy a particular product, she can directly observe similar decisions made by her friends, neighbors or coworkers. Thus, a social network can be modeled as a graph, where each vertex corresponds to an agent while edges correspond to pairwise  interactions between agents \cite{Jackson_book}. Most theoretical studies of social decision-making rest on the following basic framework \cite{Chamley_book,Jackson_book}: (i) there is some unknown parameter associated with the environment in which the network is situated; (ii) each agent receives a private signal stochastically related to this parameter; and (iii) agents select actions by aggregating their private signals with any information they receive from their neighbors in the social network. The main question is whether the agents can {\em learn} enough about this parameter of interest under the given information structure and the constraints on their information-processing capabilities. 
For instance, Acemoglu et al.~\cite{Acemoglu_etal_Bayesian_learning} consider Bayesian learning 
in dynamic social networks with randomly evolving neighborhoods, 
while Jadbabaie et al.~\cite{Jadbabaie_etal_nonBayesian_learning} examine a non-Bayesian model of learning in a fixed network, where agents form their beliefs about the underlying parameter by mixing Bayesian updates computed on the basis of their private information with the beliefs of their neighbors. 

There are several key modeling assumptions underlying these and similar works:
\begin{enumerate}
	\item[(S1)] The environment is static, meaning that the underlying parameter is drawn from a fixed probability distribution once and for all, and does not change throughout the learning process.
	\item[(S2)] Each agent has a coherent probabilistic model of the environment in the form of a joint probability measure on the Cartesian product of the parameter space and the agent's private signal space.
	\item[(S3)] The agents have no intrinsic goals or default strategies unrelated to the state of the environment. 
\end{enumerate}
In this paper, we introduce a model of discrete-time decision-making in social networks that departs from all three of these assumptions. In particular, \mr{on the descriptive level}, our setting has the following features:
\begin{enumerate}
	\item[(D1)] The environment is dynamic, and no agent has a model of its evolution. 
	\item[(D2)] In view of the item above, the environment does not admit a probabilistic representation. Instead, at each time step, each agent receives a signal that quantifies the {\em costs} of all possible actions that could be taken by this agent and its neighbors in the social network in the current state of the environment.
	\item[(D3)] No agent is compelled to take only those actions that would entail lower costs. Instead, each agent has a default mixed strategy that stays fixed regardless of the state of the environment, and must expend effort in order to deviate from this strategy.
\end{enumerate}
The distinction between the probabilistic (or Bayesian) view of the environment stipulated in S1--S2 and the nonprobabilistic view laid out in D1--D2 is along the same lines as the distinction between {\em risk} and {\em uncertainty} made in 1921 by Frank Knight \cite{Knight_book}. According to Knight, risk describes situations with outcomes modeled by random variables with known probability distributions, while uncertainty pertains to situations in which no such probabilistic description is available or even possible. For instance, uncertainty may arise due to the presence of boundedly rational agents with different sets of values, norms, and abilities. Despite the clear conceptual and practical significance of this distinction, there has been little effort in economics to formalize it mathematically. One of the few exceptions is the work of Bewley \cite{Bewley_Knightian1,Bewley_Knightian2}, who studies the behavior of a decision-making agent interacting in real time with an environment in a state of {\em punctuated equilibrium} --- i.e., intervals of (relative) stability are interrupted by ``shocks,'' corresponding to sharp and unpredictable changes. The Knightian aspect is embodied in the premise that the agent is unable to anticipate the frequency and the nature of these shocks in advance, and so may be caught by surprise. An ideal Bayesian risk-minimizing agent, on the other hand, is not really surprised by anything, since by definition it has already assigned subjective beliefs and utilities to all possible contingencies. Moreover, a Knightian agent may exhibit {\em inertia}, i.e., a tendency to stick to some default strategy unless there is a sufficiently strong signal from the environment compelling the agent to deviate from the status quo.

Thus, we are interested in the collective decision-making (or learning) capabilities of social networks in the presence of Knightian uncertainty, as captured by the assumptions D1--D3. We quantify learning in terms of {\em regret}, i.e., the difference between the realized performance of the network over a given time horizon and the best performance that could have been achieved in hindsight by a fictitious centralized entity with full knowledge of the environment's evolution. The performance criterion is induced by a time-varying sequence of composite objective functions that incorporate the total cost of actions taken by all the agents and the total effort expended by the agents in deviating from their individual default strategies. 

\subsection{A sketch of the model and a summary of results}

Let us give a more formal description. We start by considering a {\em single} agent who must choose an action from a finite set of alternatives, while attempting to balance the instantaneous cost of that action against a desire to minimize effort by sticking to some default (or \textit{status quo}) behavior. Mathematically, we may model such an agent as follows. Let $\sX$ denote the set of all possible actions, and let $\mu_0$ be a fixed probability distribution on $\sX$, where for each $x \in \sX$ we interpret $\mu_0(x)$ as the default probability that the agent will choose action $x$. (For instance, we may imagine a large population of similar agents and take $\mu_0(x)$ as the fraction of agents that tend to choose action $x$ by default.) Without loss of generality, we may suppose that $\mu_0(x) > 0$ for all $x \in \sX$. Now let $f \colon \sX \to \Reals$ be a function that prescribes the cost of each action. If we allow the agent to randomize, then a reasonable strategy for the agent would be to choose a random action according to
\begin{align*}
	\pi = \argmin_{\nu \in \cP(\sX)} \left\{ \beta \ave{\nu,f} + D(\nu \| \mu_0) \right\},
\end{align*}
where $\cP(\sX)$ is the space of all probability distributions on $\sX$,
\begin{align*}
	\ave{\nu,f} \deq \sum_{x \in \sX}\nu(x)f(x)
\end{align*}
is the expected cost of a random action sampled from the set $\sX$ according to $\nu \in \cP(\sX)$,
\begin{align*}
	D(\nu \| \mu_0) \deq \sum_{x \in \sX} \nu(x)\ln \frac{\nu(x)}{\mu_0(x)}
\end{align*}
is the {\em relative entropy} (or {\em Kullback--Leibler divergence}\footnote{The Kullback--Leibler divergence is a commonly used measure of (dis)similarity between probability distributions; we discuss some of its salient properties in Section~\ref{ssec:notation}.}) between $\nu$ and $\mu_0$ \cite{CovTho06}, and $\beta > 0$ is a parameter that controls the trade-off between loss aversion (i.e., the desire to minimize expected cost) and inertia (i.e., the desire to stick to default behavior) of a Knightian decision-maker \cite{Bewley_Knightian1}. Thus, the addition of the relative-entropy term is an analytically tractable means of penalizing excessive deviations from the default strategy. A simple argument based on the method of Lagrange multipliers gives an explicit form of the solution $\pi$:
\begin{align}\label{eq:logit}
	\pi(x) = \frac{\mu_0(x)\exp\left(-\beta f(x)\right)}{Z(\beta)},
\end{align}
where $Z(\beta) = \Ave{\mu_0,\exp(-\beta f)}$ is a normalization factor. This strategy is well-known in econometrics under the name of {\em multinomial logit choice model} \cite{McFadden_logit}, and it also plays a prominent role in the context of distributionally robust optimization \cite{HansenSargent_book}. Probability distributions of this form are also well-known in statistical physics under the name of {\em Gibbs measures} (see Section~\ref{ssec:notation} for more details), where $f$ plays the role of an energy function and $\beta$ is the {\em inverse temperature}. Note, in particular, the two extreme regimes: when $\beta = 0$ (infinite temperature), the cost $f$ has no influence on the agent, and we have $\pi = \mu_0$; on the other hand, as $\beta \to \infty$ (zero temperature), the agent has no inertia, and $\pi$ will converge to the uniform distribution supported on the set of minimizers of $f$.

In the above formulation, the agent knows the cost function $f$ and thus has no uncertainty about the consequences of various actions; we have only captured the agent's inertia by means of the relative-entropy term. Now, let us bring in an element of time and consider a boundedly rational agent operating in a dynamic environment. Bounded rationality comes from the fact that the agent is unable (or unwilling) to construct an intelligible model of its environment, in the spirit of Knightian uncertainty. The agent must take a sequence of random actions $X_1,\ldots,X_T \in \sX$ at discrete time steps $t=1,2,\ldots,T$. We suppose also that, at each time $t$, the costs of each action change unpredictably, and the agent only finds out the current cost function $f_t : \sX \to \Reals$ after having taken the action $X_t$.  However, the agent keeps track of all past cost functions $f_1,\ldots,f_{t-1}$, and may use this information when choosing $X_t$. We assume that the environment is {\em nonreactive}, i.e., the sequence $f_1,\ldots,f_T$ of instantaneous cost functions is fixed in advance. Finally, we assume that the default distribution 
$\mu_0$ over the action set $\sX$ does not change.

More formally, let $\pi_t \in \cP(\sX)$ denote the distribution of $X_t$ chosen by the agent based on all available information at time $t$. Then, the {\em instantaneous loss} incurred by the agent at time $t$ 
is given by
\begin{align}\label{eq:loss}
	\ell_t(\pi_t) \deq \beta \ave{\pi_t,f_t} + D(\pi_t \| \mu_0).
\end{align}
Due to the agent's limited forecasting ability, we adopt a backward-looking optimality criterion based on {\em worst-case regret}: If the cost functions $f_t$ are chosen from some fixed class $\cF$ known to the agent, the agent should choose a strategy (i.e., a rule for mapping all available information at each time $t$ to a probability distribution $\pi_t$ of $X_t$) so as to minimize the worst-case regret
\begin{align}\label{eq:regret}
 R_T(\cF) \deq \sup_{f_1,\ldots,f_T \in \cF} R_T(f^T),
\end{align}
where
\begin{align}\label{eq:regret_fT}
	R_T(f^T) \deq \sum^T_{t=1} \ell_t(\pi_t) - \inf_{\nu \in \cP(\sX)}\sum^T_{t=1}\ell_t(\nu)
\end{align}
is the regret with respect to a fixed sequence $f^T = (f_1,\ldots,f_T)$ of instantaneous costs. The regret quantifies the worst-case gap between the cumulative loss after $T$ time steps and the smallest cumulative loss that could have been achieved in hindsight had the agent been aware of the entire sequence $f^T=(f_1,\ldots,f_T)$ of instantaneous costs ahead of time. 
Indeed, if we normalize both sides of \eqref{eq:regret_fT} by $\beta T$, then the minimum of the per-round regret
\begin{align}\label{eq:per_round}
	\bar{R}_{\beta,T}(\cF) &= \sup_{f_1,\ldots,f_T \in \cF} \left\{\frac{1}{T}\left[\sum^T_{t=1} \ave{\pi_t,f_t} + \frac{1}{\beta} D(\pi_t \| \mu_0)\right] - \inf_{\nu \in \cP(\sX)} \left[\Ave{\nu, \frac{1}{T}\sum^T_{t=1}f_t} + \frac{1}{\beta} D(\nu \| \mu_0)\right]\right\}.
\end{align}
over all strategies for the agent quantifies the smallest gap the agent can secure in the worst case between (a) the average of the expected losses at each round without any knowledge of what will happen in future rounds and (b) the minimum expected loss the agent could attain in a single round against the \textit{empirical average} $(1/T) \sum^T_{t=1} f_t$ of the instantaneous costs. Since the agent does not make any probabilistic assumptions about the evolution of the cost sequence $f_1,\ldots,f_T$, the empirical average of this sequence is a reasonable proxy for the ``typical'' behavior of the environment. Moreover, from the results of Abernethy et al.~\cite{Abernethy09} it follows that, under certain mild conditions allowing the use of the minimax theorem, the per-round minimax regret admits the following equivalent characterization:
\begin{align}
	\bar{R}^*_{\beta,T}(\cF) &= \inf_{\text{strategies}} \bar{R}_{\beta,T}(\cF) \nonumber \\
	&= \frac{1}{T}\sup_{\tilde{\mu} \in \cP(\cF^T)} \E_{f^T \sim \tilde{\mu}} \left\{ \sum^T_{t=1} \inf_{\nu \in \cP(\sX)} \E\left[\ave{\nu,f_t} + \frac{1}{\beta}D(\nu \| \mu_0)\Bigg|f^{t-1}\right] - \inf_{\nu \in \cP(\sX)} \sum^T_{t=1} \left[\ave{\nu,f_t} + \frac{1}{\beta}D(\nu \| \mu_0)\right] \right\}, \label{eq:minimax_swap}
\end{align}
where $\cP(\cF^T)$ is the space of all probability measures on $T$-tuples over $\cF$ (with respect to a suitable $\sigma$-algebra). This characterization shows that regret minimization is a sequential (or dynamic) generalization of robust Bayesian optimization \cite{HansenSargent_book} that takes into account the fact that the agent does accumulate some information about the environment and may use it to some extent to compensate for future uncertainty about the instantaneous cost functions. Indeed, for a fixed $\tilde{\mu} \in \cP(\cF^T)$, the term corresponding to time $t$ in the first summation in \eqref{eq:minimax_swap} corresponds to the selection of the best strategy $\pi_t$ when the next instantaneous cost $f_t$ would be drawn from the posterior distribution $\tilde{\mu}(\cdot|f^{t-1})$. 

Online decision or prediction problems of this sort have received a great deal of attention in such fields as machine learning, operations research, and finance \cite{FosterVohra,CesLug06,Abernethy_etal_BlackScholes,Abernethy_markets}. Their origins date back to a seminal paper of Hannan \cite{Hannan}, who has shown that an agent making repeated decisions in a dynamic and uncertain environment will eventually ``learn'' to act almost as well as if it were aware of the sequence of environment states before beginning to act. To fix ideas, us give an illustrative example in the context of discrete optimization: the \textit{online shortest path} problem \cite{TakimotoWarmuth,KalaiVempala_online}. Let $G = (V,E)$ be a directed acyclic graph with two distinguished vertices $a$ and $b$. Let the action space $\sX$ be the set of all (directed) paths from $a$ to $b$; it can be identified with a subset of $\{0,1\}^E$, each of whose elements is a tuple $x = (x_e)_{e \in E}$, such that $x_e = 0$ or $1$ depending on whether the edge $e$ is included in the path. The amount of traffic on each edge $e \in E$ varies with time arbitrarily. If we denote the traffic on edge $e$ at time $t$ by $d_{t,e}$, then the total traffic along a given path $x \in \sX$ is given by
\begin{align}\label{eq:delay}
	f_t(x) = \sum_{e \in E} d_{t,e} x_{e}.
\end{align}
At each time $t$, the agent picks a probability distribution $\pi_t$ over paths from $a$ to $b$, takes a random path $X_t \sim \pi_t$, and experiences the average traffic of
\begin{align*}
	\ave{\pi_t,f_t} &=  \sum_{x \in \sX} \pi_t(x) \sum_{e \in E} d_{t,e} x_e.
\end{align*}
Let $\cF$ be the class of all functions of the form \eqref{eq:delay}. Let us first consider the case of no inertia ($\beta \to \infty$). Then it can be shown (see, e.g., \cite{TakimotoWarmuth} or \cite[Sec.~5.4]{CesLug06}) that the optimal strategy at time $t$ is given by
\begin{align*}
	\pi_t(x) &\propto \exp\left(-\eta \sum^{t-1}_{s=1} \sum_{e \in E} d_{s,e}x_e\right),
\end{align*}
where the parameter $\eta$ is given by ${\rm const} \cdot \sqrt{\frac{\ln |\sX|}{T}}$, resulting in $O\left(\sqrt{T\ln|\sX|}\right)$ total regret or $O\left(\sqrt{\ln |\sX|/T}\right)$ per-round regret. Note the intuitive structure of this strategy: it favors (i.e., assigns higher probabilities to) the paths consisting of edges that have experienced the smallest total traffic before time $t$. Moreover, because of the additive structure of the costs in \eqref{eq:delay}, the per-round computational cost of implementing such a strategy is $O(|E|)$ \cite{TakimotoWarmuth}. The per-round regret has an appealing interpretation as the difference between the total per-round traffic experienced by the agent over $T$ rounds and the average traffic along the best single path the agent could have chosen in hindsight knowing the average amount of traffic $(1/T)\sum^T_{t=1}d_{t,e}$ on each edge $e \in E$. If the agent has a finite inertia parameter $\beta > 0$ and a default distribution $\mu_0$ over paths from $a$ to $b$, then the optimal strategy at time $t$ would take the form
\begin{align*}
	\pi_t(x) \propto \mu_0(x) \exp\left(-\frac{\beta}{t}\sum^{t-1}_{s=1} \sum_{e \in E} d_{s,e} x_e\right)
\end{align*}
(cf.~Eq.~\eqref{eq:logit} and Section~\ref{sec:results}), reflecting the tension between the desire to stick to $\mu_0$ and the effort needed to deviate from it in order to experience less traffic.

Our main interest in this work is in the setting of online decision-making by a {\em social network} consisting of $n$ agents. This setting has the following salient characteristics:
\begin{enumerate}
\item[(i)] Each agent takes actions in a finite {\em base action space} $\{1,\ldots,q\}$, so the action space $\sX$ of the entire network is a Cartesian product $\{1,\ldots,q\}^n$. Both the number of alternatives $q$ and the number of agents $n$ are potentially very large.
\item[(ii)] The cost functions $f_t \in \cF$ decompose into sums of one- and two-variable ``local" terms, where each $q$-ary variable is associated with a separate agent. Thus, when each agent chooses an action, this action affects not only this agent, but also its neighbors in the social network.
\item[(iii)] Each of the $n$ agents receives only local information both from other agents and from the environment.
\end{enumerate}
In a large network with local communication, there are two sources of uncertainty for each agent at each time $t$: uncertainty about the future costs, as well as uncertainty about the actions of all agents outside of that agent's neighborhood in the social network.

Our main contribution is a construction of a decentralized strategy that takes into account these features and whose regret is sublinear in the time horizon $T$ and polylogarithmic in the network parameters (the number of agents and the maximum neighborhood size). We first present a {\em centralized} strategy and analyze its regret \eqref{eq:regret} with respect to a class $\cF$ of cost functions that decompose into individual (per-agent) and pairwise costs, where the latter affect only those agents that are neighbors in the social network. We then develop an approximate decentralized implementation of this centralized strategy using ideas from statistical physics (specifically, the well-known {\em Glauber dynamics} or the {\em Gibbs sampler} \cite{Glauber_Ising,Turchin_Gibbs_sampler,Tierney_MCMC}). It should be pointed out that decentralized strategies based on the Glauber dynamics have been studied in the literature on economics \cite{Blume_strategic_statmech,Young_individual_strategy_book,AlosFerrer-logit} and on evolutionary dynamics \cite{Sandholm_book} in the context of convergence to equilibrium in large systems consisting of interconnected agents with local interactions. Our main result (Theorem~\ref{thm:LI_regret_bound}) states that, under a certain regularity condition involving the inverse temperature (or inertia) parameter $\beta$ and the maximum degree of the social network, the regret of the decentralized strategy based on the Glauber dynamics also exhibits favorable scaling as a function of $T$ and network parameters. The proof of Theorem~\ref{thm:LI_regret_bound} relies on the aforementioned ideas from statistical physics, as well as on some recent developments in the theory of Markov chains --- specifically, on Ollivier's notion of a {\em positive Ricci curvature} of a Markov chain on a complete separable metric space \cite{Ollivier_Ricci}.

\subsection{Related literature}

The most closely related work to ours is a recent paper by Gamarnik et al.~\cite{Gamarnik_etal_corr_decay}, which studies decentralized combinatorial optimization of a random locally decomposable objective function and shows that, under a certain correlation decay condition similar to Dobrushin's uniqueness condition from statistical physics (see, e.g., \cite[Section~V.1]{Sim93}), it is possible to construct polynomial-time approximation schemes relying only on local information. However, that work is concerned exclusively with {\em static} (or offline) optimization problems, in which the objective function is fixed. On the other hand, just as in \cite{Gamarnik_etal_corr_decay}, we assume that the instantaneous network costs $f_1,\ldots,f_T$ decompose into a sum of individual and pairwise interaction terms, and at each time step each agent is informed only about its own cost and the pairwise costs in its immediate neighborhood. Another difference with \cite{Gamarnik_etal_corr_decay} is that they allow communication not only between any pair of immediate neighbors in the graph, but also between agents connected by paths of a given length $r \ge 1$. By contrast, we follow the rest of the social network literature and allow direct communication only between neighbors; however, there are also indirect information paths that affect the scaling of the regret.  Finally, the connection between the correlation decay conditions in \cite{Gamarnik_etal_corr_decay} and statistical physics is primarily qualitative, whereas our regularity condition stated in Theorem~\ref{thm:LI_regret_bound}, as well as the technique used in the proof of the theorem, are more directly related to ideas from statistical physics. 

There is also extensive literature on regret minimization in multiagent games, e.g., \cite{Foster_Vohra_calibrated_learning,Foster_Young_learning_Nash,Hart_MasColell_adaptive1,Hart_MasColell_adaptive2}, and in particular in graphical games \cite{Kearns_graphical_games} (a class of games, in which the payoff structure is aligned with the social network governing the agents' interactions). However, in this line of work, regret minimization is a goal of each individual agent, who views the rest of the network as a potential opponent. A typical result is that, provided each agent follows a suitable regret-minimizing strategy, the empirical distribution of the actions converges to some equilibrium (e.g., Nash or correlated equilibrium) of the game. By contrast, we view the social network as a {\em team} that has a common opponent, the environment. Thus, our work can be viewed as an extension of the classical Bayesian economic theory of teams \cite{MarschakRadner,Radner} to the realm of online decision-making in the presence of Knightian uncertainty.

\subsection{Some notation and preliminaries}
\label{ssec:notation}

Here, we provide some basic concepts and results that will be used later in the development.
The {\em total variation distance} between any two distributions $\mu,\nu \in \cP(\sX)$ is given by
\begin{align*}
	\| \mu - \nu \|_{\rm TV} \deq \frac{1}{2}\sum_{x \in \sX} |\mu(x) - \nu(x)|.
\end{align*}
The {\em Kullback--Leibler divergence} (or {\em relative entropy}) between $\mu$ and $\nu$ is
\begin{align*}
	D(\mu \| \nu) = \begin{cases}
	\displaystyle\left\langle\mu, \ln \frac{\mu}{\nu}\right\rangle & \text{if ${\rm supp}(\mu) \subseteq {\rm supp}(\nu)$},\\
	+ \infty & \text{otherwise},
\end{cases}
\end{align*}
where ${\rm supp}(\cdot)$ denotes the support of a probability distribution. The Kullback--Leibler divergence is nonnegative, i.e., $D(\mu \| \nu) \ge 0$ for all $\mu,\nu \in \cP(\sX)$, and positive definite, i.e., $D(\mu \| \nu) = 0$ if and only if $\mu = \nu$. (There are other properties, such as convexity, which we do not use in this paper. The reader is invited to consult any text on information theory, such as Cover and Thomas \cite{CovTho06}, for more details.)
 These two quantities are related via the Csisz\'ar--Kemperman--Kullback--Pinsker (CKKP) inequality \cite[Lemma~17.3.2]{CovTho06}\footnote{The inequality \eqref{eq:CKKP} is often referred to as simply Pinsker's inequality with reference to the book \cite{PinskerBook}, which was a translation of the original Russian text from 1960. However, in \cite{PinskerBook} Pinsker established a different bound that can be used to deduce \eqref{eq:CKKP}, but with a much larger constant in front of the relative entropy on the right-hand side. The tight bound \eqref{eq:CKKP} was obtained contemporaneously by Csisz\'ar \cite{Csiszar_divergence}, Kemperman \cite{Kemperman_infotheory}, and Kullback \cite{Kullback_TV,Kullback_TV_correction}. The authors would like to thank Prof.\ Sergio Verd\'u for pointing out the correct attribution.}
\begin{align}\label{eq:CKKP}
	\| \mu - \nu \|_{\rm TV} \le \sqrt{\frac{1}{2}D(\mu \| \nu)}.
\end{align}
We will also need some concepts from statistical physics (see, e.g., \cite{Sim93}). Any probability distribution 
$\mu \in \cP(\sX)$ defines a family of {\em Gibbs distributions} indexed 
by functions $g : \sX \to \Reals$:
\begin{align}\label{eq:Gibbs}
	\mu_g(x) \deq \frac{\mu(x)\exp\left(g(x)\right)}{\ave{\mu, \exp(g)}},\qquad \forall x \in \sX.
\end{align}
In statistical physics, each $x \in \sX$ is associated with a possible configuration of a physical system, and \an{$g : \sX \to \Reals$} is the negative energy function. 
\an{In that context, Eq.~\eqref{eq:Gibbs}} describes the probabilities of different configurations when the system with energy function \an{$g$} is in a state of  equilibrium with a thermal environment at unit absolute temperature \cite{Sim93}. \an{The following lemma (see, e.g., \cite[Lemma~V.1.4]{Sim93} for a slightly looser bound) provides some properties of Gibbs distributions that will be useful later on in the development of our main results}:
\begin{lemma}\label{lm:Gibbs} \an{Let $g,h$ be any two real-valued functions on $\sX$.
Then we have}
	\begin{align}
		D(\mu_g \| \mu_h) &\le \frac{\| g - h \|^2_{\rm s}}{8}, \label{eq:KL_Gibbs}
	\end{align}
	and
	\begin{align}
		\| \mu_g - \mu_h \|_{\rm TV} &\le \frac{\| g - h \|_{\rm s}}{4}, \label{eq:TV_Gibbs}
	\end{align}
where \an{$\|f\|_s$ is the {\em span seminorm} (or {\em oscillation}) of a function $f : \sX \to \Reals$
given by}
	\begin{align*}
		\| f \|_{\rm s} \deq \max_{x \in \sX} f(x) - \min_{x \in \sX}f(x).
	\end{align*}
\end{lemma}
\noindent The proof is elementary, so we give it in Appendix~\ref{app:Gibbs_lemma} for completeness.

\section{The model and problem formulation}

We model the social network by a simple undirected graph $G = (V,E)$, where each vertex $v \in V$ is associated with an agent, and the edges $\edge{u}{v} \in E$ indicate symmetric pairwise interactions (in particular, information exchange) among agents. For each $v$, we denote by $\partial v \deq \big\{ u \in V: \edge{u}{v} \in E \big\}$ the set of {\em neighbors} of $v$, and let $\partial_+v \deq \{v\} \cup \partial v$ denote the set consisting of agent $v$ and all of its neighbors. The {\em maximum degree} of $G$ is
\begin{align*}
	\Delta \deq \max_{v \in V}|\partial v|.
\end{align*}
Each agent takes actions in the base \an{action} set $\{1,\ldots,q\}$. 
The elements of the set \begin{align*}
\sX \deq \Big\{ x = (x_v)_{v \in V} : x_v \in \{1,\ldots,q\}\Big\}
\end{align*} 
will be referred to as {\em network action profiles}. 
For each $v \in V$, we fix a probability measure \an{$\mu_{v,0}$} on $\{1,\ldots,q\}$, and let $\mu_0 \in \cP(\sX)$ denote the product measure
\begin{align}\label{eq:default_measure}
	\mu_0(x) \deq \prod_{v \in V}\an{\mu_{v,0}}(x_v), \qquad \hbox{\an{for all} $x \in \sX$}.
\end{align}
We assume that each \an{$\mu_{v,0}$} charges every action $a \in \{1,\ldots,q\}$, i.e., $\an{\mu_{v,0}}(a) > 0$ for all $a$. The probability measure \an{$\mu_{v,0}$} describes the default individual behavior of agent $v$. Finally, we are given two classes of local cost functions: the class $\Phi$ of one-variable (vertex) costs $\phi : \{1,\ldots,q\} \to \Reals$ and the class $\Psi$ of two-variable (edge) costs $\psi : \{1,\dots,q\} \times \{1,\ldots,q\} \to \Reals$. With this, we denote by $\cF = \cF_{\Phi,\Psi}$ the space of all functions $f \colon \sX \to \Reals$ of the form
\begin{align}\label{eq:cost_function}
	f(x) = \sum_{v \in V} \phi_v(x_v) + \sum_{\edge{u}{v} \in E}\psi_{u,v}(x_u,x_v),
\end{align}
where $\phi_v \in \Phi$ and $\psi_{u,v} \in \Psi$ for all $v \in V$ and all $ \edge{u}{v} \in E$. 

The interaction among the agents and the environment takes place according to the following protocol: Initially, each agent $v \in V$ starts out with an empty information set $I_{v,0} = \varnothing$ and draws an action $X_{v,0} \in \{1,\ldots,q\}$ at random according to \an{$\mu_{v,0}$}, independently of all other agents. At each discrete time step $t \in \{1,\ldots,T\}$, a single agent $U_t \in V$ is activated uniformly at random independently of all other past data. 
This agent takes a random action $X_{U_t,t}$ on the basis of all information currently available to it, 
while all other agents $v \in V \backslash \{U_t\}$ replay their actions from the previous time step $t - 1$. Once the network action profile $X_t = (X_{v,t} : v \in V)$ for time $t$ is generated, each agent $v$ observes
\an{its instantaneous cost function $\phi_{v,t}$, its instantaneous local-interaction cost functions 
$\psi_{u,v,t}$ for $u \in \partial v$, and the decisions of all its neighbors (of course, the agent knows its own decision $X_{v,t}$). 
Formally, each agent $v\in V$ at time $t$ observes $\iota_{v,t}$, where } 
\begin{align}\label{eq:new-info}
	\iota_{v,t} = \Big( \phi_{v,t}; \big(\psi_{u,v,t} : u \in \partial v\big);  \big(X_{u,t} : u \in \partial_+v\big)\Big),
\end{align}
and updates its information to $I_{v,t} = (I_{v,t-1},\iota_{v,t})$. Here, $(\phi_{v,t})_{v \in V}$ and $(\psi_{u,v,t})_{\edge{u}{v} \in E}$ are the local costs for each agent and for each pair of interacting agents that the environment has generated for time $t$. \an{As we mentioned earlier}, 
we assume that the environment is nonreactive, i.e., all the cost functions are fixed in advance but revealed to the agents sequentially. Figure~\ref{fig:ODO} gives a summary of this process.

\begin{figure}[t]
\bookbox{\small
{\textsf{Parameters:}} base action set $\{1,\ldots,q\}$; network graph $G = (V,E)$; 
default probability measures \an{$\mu_{v,0}$} for all $v \in V$; local function classes $\Phi,\Psi$; 
number of rounds $T \in 
\Naturals$.

\medskip\noindent
{\textsf{Initialization of information sets:}} for each $v \in V$, let $I_{v,0} = \varnothing$.

\medskip\noindent
\textsf{Initialization of actions:} for each $v \in V$, draw $X_{v,0}$ at random according to $\an{\mu_{v,0}}$, independently of all other $v$'s.

\medskip\noindent
For each round $t=1,2,\ldots,T$:
\begin{itemize}
\item[(1)] An agent $U_t \in V$ is chosen uniformly at random.
\item[(2)] Agent $U_t$ draws a random action $X_{U_t,t}$ on the basis of \an{its} current information $I_{U_t,t-1}$; all other agents $v \in V\backslash \an{\{U_t\}}$ replay their most recent action: 
$X_{v,t} = X_{v,t-1}$.
\item[(3)] Each agent $v \in V$ observes
\begin{itemize}
\item the current cost functions $\phi_{v,t} \in \Phi$ and $\psi_{u,v,t} \in \Psi$ for all $u \in \partial v$
\item the actions $X_{u,t}$ for all $v \in \partial_+v$,
\end{itemize}
and updates \an{its} information set to 
\an{$I_{v,t}=(I_{v,t-1},\iota_{v,t})$, where $\iota_{v,t}$ is the new information available to agent at time $t$ (cf.\ Eq.~\eqref{eq:new-info}).
}
\end{itemize}
}
\caption{Online discrete optimization in a network \an{of agents} with local interactions.}
\label{fig:ODO}
\end{figure}

For each $t = 1,\ldots,T$, let $\mu_t$ denote the probability distribution of the network action profile $X_t$.\footnote{We will adhere to the following convention: we will use $\mu_t$ (respectively, $\pi_t$) to denote the distribution of the action profile $X_t$ in the decentralized (respectively, centralized) scenario.} For a fixed sequence of cost functions selected by the environment, the probability measures $\mu_1,\ldots,\mu_T$ are fully specified given the initial condition $\mu_0$ in \eqref{eq:default_measure} and the sequence of conditional probability distributions
\begin{align}\label{eq:transitions}
	\Prob_{t+1}(x_{t+1}|I_{t}) = \frac{1}{|V|}\sum_{v \in V}\Prob_{v,t+1}(x_{v,t+1}|I_{v,t})\1_{\{x_{-v,t+1}=x_{-v,t}\}}, \qquad t = 0,1,\ldots,T-1
\end{align}
where $\1_{\{\cdot\}}$ is an indicator function that takes value $1$ when the logical predicate $\{\cdot\}$ is true and is $0$ otherwise, $I_{t} = (I_{v,t} : v \in V)$ is all \an{the information available immediately after time $t$}, $\Prob_{v,t+1}(\cdot|I_{v,t})$ is the conditional distribution (or local stochastic update rule) according to which agent $v$ draws its action $x_{v,t+1}$, \an{while} $x_{-v,t}$ is the $(|V|-1)$-tuple obtained from $x_t$ by deleting the coordinate corresponding to agent $v$, \an{i.e., }$x_{-v,t} \deq \big(x_{u,t} : u \in V \backslash \{v\}\big)$. The instantaneous loss incurred by the network at time $t$ is given by
\begin{align*}
	\ell_t(\mu_t) = \beta\ave{\mu_t,f_t} + D(\mu_t \| \mu_0),
\end{align*}
where
\begin{align*}
	f_t(x) = \sum_{v \in V}\phi_{v,t}(x_v) + \sum_{\edge{u}{v} \in E}\psi_{u,v,t}(x_u,x_v)
\end{align*}
is the instantaneous cost function for the entire network at time $t$. After $T$ rounds, the {\em regret} of the network with respect to the sequence $f_1,\ldots,f_T$ is
\begin{align*}%\label{eq:decentralized_regret}
	R^{\rm LI}_T(f^T) \deq  \sum^T_{t=1} \ell_t(\mu_t) - \inf_{\nu \in \cP(\sX)}\sum^T_{t=1}\ell_t(\nu)
\end{align*}
where the superscript ${\rm LI}$ stands for ``local interaction.'' The corresponding worst-case regret is
\begin{align}\label{eq:decentralized_regret}
	R^{\rm LI}_T(\cF) \deq \sup_{f_1,\ldots,f_T \in \cF} R^{\rm LI}_T(f^T).
\end{align}
 Our objective is to design the local stochastic update rules 
 $\Prob_{v,t}(\cdot|I_{v,t})$ for all $v \in V$ and all $t \in \{1,\ldots,T\}$ to guarantee that 
 the regret~\eqref{eq:decentralized_regret} is sublinear in $T$ and polynomial in 
the inverse temperature parameter $\beta$, the number of basic actions $q$, the size $|V|$ of the network, 
 and the maximum number $\Delta$ of each agent's neighbors in the social graph.

\section{The main results}
\label{sec:results}

To motivate our design \an{of a decentralized strategy}, we start by developing a particular centralized scheme that, as we shall see, can be well-approximated with a natural distributed implementation. Consider a fixed but arbitrary sequence of instantaneous cost functions $f_1,\ldots,f_T \in \cF$ chosen by the environment. Our centralized strategy is obtained by the following recursive construction. Suppose that the distributions $\pi_1,\ldots,\pi_t$ of the action profiles $X_1,\ldots,X_t$ have already been chosen. We choose the next $\pi_{t+1}$ to balance the greedy tendency to minimize the most recent instantaneous loss $\ell_t(\cdot) = \beta\ave{\cdot,f_t} + D(\cdot \| \mu_0)$ against the cautious tendency to stay close to what worked well in the past, i.e., $\pi_t$. Hence, a good candidate for $\pi_{t+1}$ is
\begin{align}\label{eq:OLP_objective}
	\pi_{t+1} = \argmin_{\pi \in \cP(\sX)}\Bigg\{ \gamma_t \Big[ \beta\ave{\pi,f_t} + D(\pi \| \mu_0)\Big] + D(\pi \| \pi_t)\Bigg\},
\end{align} 
where the weight $\gamma_t>0$ controls the trade-off between the greedy and the cautious behavior. This construction is reminiscent of the so-called {\em mirror descent algorithms} for online convex optimization~\cite[Chapter~11]{CesLug06}, with $\gamma_t$ viewed as a step size at time $t$. In contrast to the mirror descent algorithms, here the optimization is performed in the space of probability measures and there is no linearization of the objective function. 
An application of the method of Lagrange multipliers leads to the following solution:
\begin{align}\label{eq:Renyi_strategy}
  \pi_1 = \mu_0\quad\hbox{and} \quad
		\pi_{t+1}(x) = 
		\frac{\left(\mu^{\gamma_t}_0(x) \pi_t(x) 
		\exp\left(-\gamma_t \beta f_t(x)\right)\right)^{\frac{1}{1+\gamma_t}}}{Z_{t+1}}, \,\, t = 1,2,\ldots
	\end{align}
where $Z_{t+1}$ is the normalization constant ensuring that $\pi_{t+1}$ 
is a bona fide probability distribution. We will work with $\gamma_t=\frac{1}{t}$.
We now summarize the key properties of this strategy.

\begin{proposition}\label{prop:OLP} 
The strategy \eqref{eq:Renyi_strategy}, 
\an{with $\gamma_t=\frac{1}{t}$}, has the following properties:
	\begin{enumerate}
		\item For any  $t = 0,1,2,\ldots$, the distribution $\pi_{t+1}$ can be expressed as
		\begin{align}\label{eq:pi_t}
			\pi_{t+1}(x) = \frac{\mu_0(x) 
			\exp\left(-\beta F_t(x)\right)}{\tZ_{t+1}},
		\end{align}
		where $\tZ_{t+1}$ is a normalization constant, %$F_0=0$% 
		and the functions $F_t$ are given by 
		\begin{align}\label{eq:discount-fun}
		F_t(x)= \begin{cases}
		0, &t=0, \\
		\frac{1}{t+1}\sum_{s=1}^{t} f_s(x), &t \ge1.
	\end{cases}
		\end{align} 
		\item Suppose that the functions $\phi \in \Phi$ and $\psi \in \Psi$ take values 
		in the interval $[-1,1]$. Then, for all $t \ge 1$,
		\begin{align}\label{eq:KL_stepsize}
		D(\pi_t \| \pi_{t+1}) \le 2\left(\frac{\beta|V|(\Delta+1)}{t+1}\right)^2.
		\end{align}
		\item It achieves the following bound on the worst-case regret: for all $T\ge1$,
				\begin{align}\label{eq:Renyi_regret}
					R_T(\cF)  \le 
					2\left(\beta|V|(\Delta+1) \right)^2 \ln (T+1)
				+ \ln \frac{1}{\theta},
				\end{align}
				where
				$\cF=\cF_{\Phi,\Psi}$ and
				\begin{align*}
					\theta \deq \min_{a \in \{1,\ldots,q\}} \mu_{0}(a).
				\end{align*}
	\end{enumerate}
\end{proposition}

Several observations and remarks are in order:
\begin{enumerate}
	\item The $O(\ln T)$ scaling of the regret is a consequence of the general fact that, due to the presence of the relative entropy term, the loss functions $\ell_t$ are \textit{strongly convex} with respect to the total variation norm (see, e.g., \cite{BarHazRak08}).
	In fact, there are matching lower bounds \cite{ABRT08} that show that this scaling of the regret is optimal for strongly convex losses. However, our analysis of the centralized strategy in \eqref{eq:Renyi_strategy} is self-contained, and the three parts of Proposition~\ref{prop:OLP} reflect the logical structure of the proof: first, we show that the strategy at time $t+1$ is given by a Gibbs measure determined by the empirical average of the instantaneous costs revealed up to time $t$; then we show that the difference between the strategies at successive times $t$ and $t+1$, as measured by the relative entropy $D(\pi_t \| \pi_{t+1})$, decays rapidly with $t$; and finally we use these intermediate results to derive a bound on the overall regret. Thus, even though the result in Proposition~\ref{prop:OLP} is not new, it differs from the rest of the literature by its emphasis on the \textit{dynamical properties} of the strategy in \eqref{eq:Renyi_strategy}, which we will exploit in the proof of our main result.
	\item Implementation of the above centralized strategy does not require advance knowledge of the time horizon $T$.	
	\item Inspection of the non-recursive 
	form \eqref{eq:pi_t} of $\pi_{t+1}$ sheds light on the role of the decaying factor $1/t$ in \eqref{eq:OLP_objective}: it is used to {\em dampen} the influence of past instantaneous costs $f_1,\ldots,f_t$. In particular, at time $t$, each cost function enters into the strategy $\pi_t$ with the same weight $1/(t+1)$. As we will see later, this averaging is crucial in ensuring that we can approximate each global randomized strategy $\pi_t$ using purely local update rules.

	\item \an{From the standpoint of the influence of the initial distributions $\mu_{v,0}$, 
	the regret bound in~\eqref{eq:Renyi_regret}
	is minimized when $\mu_{v,0}$ is the uniform distribution for all $v\in V$, resulting in
	\begin{align*}
	      R_T(\cF) \le 2\left(\beta|V|(\Delta+1) \right)^2\ln(T+1)	+ |V| \ln q.
       \end{align*}}
	\item \mr{The fact that the regret is proportional to $(|V|(\Delta+1))^2$ is not surprising in light of the fact that each cost function $f_t$ is a sum of $|V|(\Delta+1)$ terms, each of which is bounded by $1$. Thus, $\| f_t \|_s = O\big(|V|(\Delta+1)\big)$. Since the regret is governed by the relative-entropy drift terms $D(\pi_t \| \pi_{t+1})$, by Lemma~\ref{lm:Gibbs} we expect it to scale with $\max_t\| f_t \|_s = O\big((|V|(\Delta+1)^2\big)$.}
\end{enumerate}
We now use this centralized scheme to construct appropriate local update rules 
$\Prob_{v,t+1}(\cdot|I_{v,t})$ for all $v \in V$ and \an{$t \in \{1,\ldots,T-1\}$}. 
For any cost function $f$ of the form \eqref{eq:cost_function}, any $v \in V$, and any {\em boundary condition} $x_{\partial v} \in \{1,\ldots,q\}^{|\partial v|}$, we define the local cost at $v$ by
\begin{align*}
	f_v(a,x_{\partial v}) \deq \phi_v(a) + \sum_{u \in \partial v} \psi_{u,v}(x_u,a), \qquad \forall a \in \{1,\ldots,q\}.
\end{align*}
Similar notation will be used for time-indexed instantaneous and discounted cumulative costs, 
i.e., $f_{v,t}$ based on $f_t$, and $F_{v,t}$ based on $F_t$. 
For each $v \in V$ and each $t$, let
\begin{align}\label{eq:one_site_update}
	\Prob_{v,t+1}(x_{v,t+1}|I_{v,t}) \deq \frac{\mu_{v,0}(x_{v,t+1})\exp\left(-\beta F_{v,t}(x_{v,t+1},x_{\partial v,t})\right)}{Z_{t+1}(x_{\partial v,t})},
\end{align}
where 
\begin{align}\label{eq:costF}
F_{v,t}=\frac{1}{t+1}\sum_{s=1}^t f_{v,s};
\end{align}
The normalization constant $Z_{t+1}(x_{\partial v,t})$ in~\eqref{eq:one_site_update} 
now depends on the action profile $x_{\partial v,t}$ of the neighborhood of agent $v$ after time $t$:
\begin{align*}
	Z_{t+1}(x_{\partial v,t}) 
	   = \sum_{a \in \{1,\ldots,q\}} \mu_{v,0}(a)\exp\left( -\beta F_{v,t}(a,x_{\partial v,t})\right).
\end{align*}
Note that the history of previous local action profiles $x_{\partial v,1},\ldots,x_{\partial v,t}$ enters into 
the conditional probabilities  \eqref{eq:one_site_update} only through 
the most recent action profile $x_{\partial v,t}$. Moreover, if we consider a fixed but arbitrary sequence of network costs $f_1,\ldots,f_T$, then we may simplify our notation by suppressing the dependence of 
the transition probabilities $\Prob_{v,t+1}(\cdot|I_{v,t})$ and $\Prob_{t+1}(\cdot|I_t)$ on 
the costs $f_1,\ldots,f_t$ and past action profiles $x_1,\ldots,x_{t-1}$. 
Thus, instead of $\Prob_{v,t+1}(x_{t+1}|I_{v,t})$, we will write
\begin{align}\label{eq:Glauber}
	\Prob_{t+1}(x_{t+1}|x_t) = \frac{1}{|V|}\sum_{v \in V}\Prob_{v,t+1}(x_{v,t+1}|x_{\partial v,t})\1_{\{ x_{-v,t+1} = x_{-v,t} \}},
\end{align}
where we have also used the same convention for the local update rules $\Prob_{v,t+1}(\cdot|I_{v,t})$. 
So, when the instantaneous costs $f_1,\ldots,f_T$ are fixed, the network action profiles $X_0,X_1,\ldots,X_T$ 
form a Markov chain with initial distribution $\mu_0$ and 
time-inhomogeneous transition probabilities 
\begin{align*}
	{\rm Pr}(X_{t+1}=y|X_t=x) = \Prob_{t+1}(y|x).
\end{align*}
One can recognize the Markov transition kernel $\Prob_{t+1}(x_{t+1}|x_t)$ constructed from \eqref{eq:one_site_update} according to \eqref{eq:transitions} as one step of the {\em Glauber dynamics} (or the {\em Gibbs sampler}) \cite{Glauber_Ising,Turchin_Gibbs_sampler,Tierney_MCMC} induced by the Gibbs distribution $\pi_{t+1}$ in \eqref{eq:pi_t}. Consequently, for each $t$ we have the detailed balance (or time-reversibility) property
\begin{align}\label{eq:reversibility}
	\pi_t(x) \Prob_t(y|x) = \pi_t(y) \Prob_t(x|y) , \qquad \forall x,y \in \sX
\end{align}
which implies that $\pi_t$ is an invariant distribution of $\Prob_t$ 
(we give a self-contained proof of this fact in Appendix~\ref{app:Gibbs_sampler}). 
In mathematical economics and game theory, the Glauber dynamics was used by Blume \cite{Blume_strategic_statmech} (under the name ``log-linear learning") and by Young \cite{Young_individual_strategy_book} (under the name ``spatial adaptive play") 
to model the emergence of optimal global behavior in networks of agents with local interactions; see also a recent paper by Al\'os-Ferrer and Netzer \cite{AlosFerrer-logit} for a discussion of a more general class of {\em logit-response dynamics} that includes the Glauber dynamics as a special case.

\mr{We now state our main result: a bound on the regret of the decentralized strategy based on the Glauber dynamics.}

\anr{
\begin{theorem}\label{thm:LI_regret_bound} 
Suppose that all functions in $\Phi$ and $\Psi$ take values in $[-1,1]$. Also,
suppose that the parameter $\beta>0$ and satisfies the following condition:
\begin{align}\label{eq:regularity_conditions}
\Delta \beta < 1.
\end{align}
Then the strategy \eqref{eq:one_site_update}--\eqref{eq:Glauber} based on the Glauber dynamics attains the following worst-case regret:
	\begin{align}\label{eq:LI_regret_bound}
		R^{\rm LI}_T(\cF) \le
			R^{\rm LI}_T(f^T) &\le \frac{|V|}{1-\Delta \beta}\left( 2\beta^2|V|^3(\Delta+1)^2 +  K\left(|V| \ln \frac{q^2}{\theta_d} + \ln T\right)\right)\ln (T+1) \nonumber\\
				& \qquad + 2\left(\beta|V|(\Delta+1) \right)^2\ln(T+1)	+ T_1 |V| \ln \frac{q^2}{\theta_d} + \frac{2 \beta|V|^3(\Delta+1)}{1-\Delta\beta}  + \ln \frac{1}{\theta},
	\end{align}
	where $K \deq \max \left\{ |V|, \beta |V|^2 (\Delta+1)\right\}$,
	$$
	\theta_d \deq \min_{v \in V}\min_{a \in \{1,\ldots,q\}} \mu_{v,0}(a),
	$$
and $T_1 \equiv T_1(\beta,\Delta,|V|)$ is a positive constant independent of $T$.
\end{theorem}
}
\noindent A few comments on the interpretation of the above result:
\begin{enumerate}
	\item The regularity condition in \eqref{eq:regularity_conditions} is needed to ensure that the sequence of action profile distributions $\mu_1,\ldots,\mu_T$ induced by the Glauber dynamics \eqref{eq:one_site_update}--\eqref{eq:Glauber} closely tracks its centralized counterpart $\pi_1,\ldots,\pi_T$ from Proposition~\ref{prop:OLP}. It is, essentially, a Dobrushin uniqueness condition from statistical physics (see, e.g., \cite[Section~V.1]{Sim93}), which is typically used to establish rapid mixing (i.e., convergence to the invariant distribution) of the Glauber dynamics \cite{Weitz_Dorbushin,Dyer_etal_mixing}. The correlation decay conditions driving the results of Gamarnik et al.~\cite{Gamarnik_etal_corr_decay} are similar in spirit. 
	
\item We note that $\beta$ is a fixed exogenous parameter that quantifies the responsiveness of agents to changes in their environment (as reflected through the time-varying instantaneous costs). The regularity condition in \eqref{eq:regularity_conditions} therefore involves only the intrinsic parameters of the network and says that the Glauber dynamics \eqref{eq:one_site_update}--\eqref{eq:Glauber} is mixing whenever the temperature parameter $1/\beta$ is larger than the maximum number of neighbors of any agent. 

\item 	Ignoring terms of order lower than $\ln T$, we can express the regret bound more succinctly as
	\begin{align}\label{eq:LI_regret}
		R^{\rm LI}_T(\cF) = O\left(\frac{|V|^4 (\Delta\beta)^2}{1-\Delta\beta}\left(\ln \frac{q}{\theta_d}\right)\left(\ln T\right)^2\right).
	\end{align}
Compared to the regret bound \eqref{eq:Renyi_regret} in the centralized case, 
	the local-information regret \eqref{eq:LI_regret} has a worse (but still polynomial) dependence on the network parameters. We also observe that the regret now scales as $(\ln T)^2$, as opposed to the optimal centralized scaling of $\ln T$.
\end{enumerate}

\section{Proofs}
\subsection{Proof of Proposition~\ref{prop:OLP}}
\paragraph{Part 1.} 
We analyze the strategy~\eqref{eq:Renyi_strategy} with $\gamma_t=\frac{1}{t}$ for all $t\ge 1$. The proof is by induction on $t$. The base case is $t=1$, for
which, using \eqref{eq:Renyi_strategy} and the fact that $\gamma_1=1$, we have for all $x\in\sX$,
     \[ \pi_{2}(x) = 
		\frac{\left(\mu_0(x) \pi_1(x) 
		\exp\left(-\beta f_1(x)\right)\right)^{\frac{1}{2}} }{Z_{2}}. \]
Since $\pi_1=\mu_0$, it follows that 
\begin{align*}
\pi_2 = \frac{\mu_0 \exp\left(-\frac{\beta}{2} f_1\right)}{\tZ_2}, 
\end{align*}
thus showing that the expressions in~\eqref{eq:pi_t} and \eqref{eq:discount-fun} are 
valid for $t=1$ with $\tZ_2=Z_2$.  Now, suppose that the strategy $\pi_{t+1}$ satisfies \eqref{eq:pi_t} and \eqref{eq:discount-fun} for a given $t+1$. Then, according to the definition of $\pi_{t+2}$ via~\eqref{eq:Renyi_strategy},  and using the fact that $\gamma_{t+1}=\frac{1}{t+1}$, we have
	\begin{align*}
		\pi_{t+2}(x) 
		&= \frac{ \left(\mu^{\frac{1}{t+1}}_0(x) \pi_{t+1}(x)
		   \exp\left(- \frac{\beta}{t+1} f_{t+1}(x)\right) \right)^{\frac{t+1}{t+2} }}{ Z_{t+2}} \\
		&= \frac{\left(\mu^{\frac{1}{t+1}} _0(x) 
		       \mu_0(x) 
			\exp\left(-\frac{\beta}{t+1}\sum_{s=1}^{t} f_s(x)\right)
			\exp\left(-\frac{\beta}{t+1} f_{t+1}(x)\right)
			\right)^{\frac{t+1}{t+2}}}
			{\tZ_{t+1}^{\frac{t+1}{t+2}} Z_{t+2}},
	\end{align*}
where the last equality follows from the induction hypothesis.  Hence, for all $x\in\sX$, we have
\begin{align}\label{eq:th1part1-1}
		\pi_{t+2}(x) 
		&=  \frac{\mu_0(x) 
			\exp\left(-\frac{\beta}{t+2}\sum^{t+1}_{s=1}f_s(x)\right)}
			{\tZ_{t+1}^{\frac{t+1}{t+2}} Z_{t+2}}.
\end{align}
Eq.~\eqref{eq:th1part1-1} shows that
that~\eqref{eq:pi_t} and \eqref{eq:discount-fun} hold for $t+2$, 
with  $\tZ_{t+2} = \tZ_{t+1}^{\frac{1}{1+\gamma_{t+1}}}Z_{t+2}$.
Hence, \eqref{eq:pi_t} and \eqref{eq:discount-fun} are valid for all $t\ge1$.

\anr{
\paragraph{Part 2.} In order to bound the terms $D(\pi_t \| \pi_{t+1})$, it is convenient to use the form \eqref{eq:pi_t} of $\pi_t$ from Part 1, which expresses $\pi_t$ as a Gibbs measure. 
Therefore, we can use Lemma~\ref{lm:Gibbs} to get
\begin{align}\label{eq:start}
	D(\pi_t \| \pi_{t+1}) \le \frac{ \beta^2 \Big\| F_t - F_{t-1}\Big\|^2_{\rm s}}{8}.
\end{align}
Now we need to bound the span seminorm of $F_t - F_{t-1}$. 
From the definition of $F_t$ in \eqref{eq:discount-fun}
and the relation~\eqref{eq:th1part1-1},
it can be seen that  
\begin{align}\label{eq:F-diff}
F_1(x)=\frac{1}{2} f_1(x)\quad\hbox{and} \quad
F_t(x)=\frac{1}{t+1} f_t(x) + \frac{t}{t+1} F_{t-1}(x)
 \ \hbox{ for all $t\ge 2$}.
\end{align}
Since $F_0=0$, we can write
\begin{align*}
	F_t(x) - F_{t-1}(x) = \frac{1}{t+1}\left[f_t(x) - F_{t-1}(x)\right]
	\quad\hbox{for all }t\ge1.
\end{align*}
Hence,
\begin{align}\label{eq:span_bound}
       \Big\| F_1 - F_{0} \Big\|_{\rm s} \le \frac{1}{2} 
	\left\|  f_1 \right\|_{\rm s}\quad\hbox{and}\quad
	\Big\| F_t - F_{t-1} \Big\|_{\rm s} \le \frac{1}{t+1} 
	\left( \left\|  f_t \right\|_{\rm s} + \Big\| F_{t-1} \Big\|_{\rm s} \right)\qquad\hbox{for all $t\ge2$}.
\end{align}
Now, using the definition of $F_{t-1}$ and the relation $\|f\|_{\rm s}\le 2\|f\|_\infty$,
valid for any function $f$ on $\sX$, we have
\begin{align*}
	\bigg\| F_{\ell} \bigg\|_{\rm s} \le 
	\frac{1}{\ell+1}\sum^{\ell}_{s=1} \|f_s\|_{\rm s} 
	\le \frac{2}{\ell+1}\sum^{\ell}_{s=1}\|f_s\|_\infty.
\end{align*}	
By employing Lemma~\ref{lm:sup_norm_bound}, according to which $\|f_s\|_\infty \le |V|(\Delta+1)$,
we further obtain
\begin{align}\label{eq:discounted_span_norm_bound}
	\bigg\| F_{\ell} \bigg\|_{\rm s} \le 
	2|V|(\Delta+1)
	\qquad \hbox{ for all }\ell\ge1.
\end{align}
Using~\eqref{eq:discounted_span_norm_bound} in the expression~\eqref{eq:span_bound}, we get
\begin{align*}
	\Big\| F_t - F_{t-1} \Big\|_{\rm s} 
	\le \frac{4|V|(\Delta+1)}{t+1}.
\end{align*}
By substituting the preceding estimate into~\eqref{eq:start} we obtain, for all $t\ge1$,
$$
D(\pi_t \| \pi_{t+1}) \le \frac{ \beta^2} {8} 
\left( \frac{4|V|(\Delta+1)}{t+1} \right)^2
=2\left(\frac{\beta|V|(\Delta+1)}{t+1} \right)^2,
$$
which gives us \eqref{eq:KL_stepsize}.
}
\anr{
\paragraph{Part 3.} We show the bound for the regret $R_T(f^T)$ by first writing down
an {\em exact} expression for it and then developing a suitable upper bound. For every $t$, we can use the definition \eqref{eq:Renyi_strategy} of $\pi_{t+1}$ to write
	\begin{align*}
		\beta f_t(x) 
		&= \ln \mu_0(x) + t\ln \pi_t(x) 
		    - \left( t + 1\right)\left( \ln Z_{t+1} + \ln \pi_{t+1}(x)\right).
	\end{align*}
Therefore, for any $\nu \in \cP(\sX)$,
\begin{align*}
	\ell_t(\nu) 
	&= \beta\ave{\nu,f_t} + D(\nu \| \mu_0) \\
	&=  \Ave{\nu,\beta f_t + \ln \frac{\nu}{\mu_0}} \\
	&= \Ave{\nu, \ln \mu_0 
		+t\ln \pi_t - \left(t+1\right)\left( \ln Z_{t+1} 
		+ \ln \pi_{t+1}\right) + \ln \frac{\nu}{\mu_0}} \\
	&= \Ave{\nu, t \ln \frac{\pi_t}{\pi_{t+1}} 
	       + \ln \frac{\nu}{\pi_{t+1}}} - \left(t+1\right)\ln Z_{t+1} \\
	&= t\Ave{\nu, \ln \frac{\nu}{\pi_{t+1}} - \ln \frac{\nu}{\pi_t}} 
		+ \Ave{\nu, \ln \frac{\nu}{\pi_{t+1}}} - \left(t+1\right)\ln Z_{t+1}\\
	&= \left[ \left(t+1 \right)D(\nu \| \pi_{t+1}) - t D(\nu \| \pi_t)\right] 
		- \left( t+1\right) \ln Z_{t+1}.
\end{align*}
In particular, letting $\nu = \pi_t$ and using the fact that $D(\nu \| \nu) = 0$ for all $\nu$, we get
\begin{align*}
	\ell_t(\pi_t) &=  \left(t+1 \right)\left[D(\pi_t \| \pi_{t+1}) - \ln Z_{t+1}\right].
\end{align*}
Therefore, summing from $t=1$ to $t=T$ and telescoping, we obtain
\begin{align*}
 \sum^T_{t=1} \left[\ell_t(\pi_t) -  \ell_t(\nu) \right] 
 &= \sum^T_{t=1} \left(t+1\right) D(\pi_t \| \pi_{t+1}) 
 + \sum^T_{t=1}\left[t D(\nu \| \pi_t) 
         - \left(t+1\right)D(\nu \| \pi_{t+1})\right] \\
 &= \sum^T_{t=1} \left(t+1\right) D(\pi_t \| \pi_{t+1}) 
 + D(\nu \| \pi_1)  - (T+1) D(\nu \| \pi_{T+1}).
\end{align*}
Using the fact that $\pi_1 = \mu_0$ and $D(\nu \| \pi_{T+1})\ge0$, we can further bound the regret as follows:
\begin{align}\label{eq:Renyi_regret_bound}
	R_T(f^T) \le 
	\sum^T_{t=1} \left(t+1\right) D(\pi_t \| \pi_{t+1}) 
	+ \ln \frac{1}{\theta},
\end{align}
where we have used the inequality 
\begin{align*}
	D(\nu \| \mu_0) = \Ave{\nu, \ln \frac{\nu}{\mu_0}} \le |V| \ln \frac{1}{\theta}
\qquad\hbox{ for any $\nu \in \cP(\sX)$,}
\end{align*}
which holds since $\mu_0(x) > 0$ for all $x \in \sX$.
Upon substituting the bound for  $D(\pi_t \| \pi_{t+1})$ from Part 2 into~\eqref{eq:Renyi_regret_bound},
we obtain
\begin{align*}
R_T(f^T) &\le 
	\sum^T_{t=1} 2\left(t+1\right)\left(\frac{\beta|V|(\Delta+1)}{t+1}\right)^2
		+ \ln \frac{1}{\theta}\cr
		&=2\left(\beta|V|(\Delta+1) \right)^2
		\sum^T_{t=1} \frac{1}{t} 
		+ \ln \frac{1}{\theta} \\
		&\le  2\left(\beta |V|(\Delta+1)\right)^2 \ln (T+1) + \ln \frac{1}{\theta},
\end{align*}
where the last inequality follows from 
\[\sum^T_{t=1} \frac{1}{t+1}\le\int_1^{T+1}\frac{\d t}{t} = \ln (T+1).\]
Since this bound holds uniformly in all $f_1,\ldots,f_T \in \cF$, we get~\eqref{eq:Renyi_regret}.
}

\subsection{Proof of Theorem~\ref{thm:LI_regret_bound}}

Before proceeding with the formal proof, let us briefly outline the intuition behind it. \an{The underlying idea is to express the regret $R^{\rm LI}_T(f^T)$ for the decentralized local-interaction strategy $\{\mu_t\}^T_{t=0}$ as the sum of the regret $R_T(f^T)$ for the centralized strategy $\{\pi_t\}^T_{t=0}$ and the extra cost due to decentralization. Theorem~\ref{prop:OLP} provides a bound for $R_T(f^T)$, and the main effort of the proof is in establishing a bound on the total decentralization cost incurred over the time horizon $T$. In turn, this total decentralization cost depends on the distances between the centralized action profile distribution $\pi_t$ and its centralized counterpart $\mu_t$ for $t=1,\ldots,T$. These distances turn out to be small due to the use of the Glauber dynamics.}

More specifically, consider the centralized strategy $\{\pi_t\}^{T+1}_{t=0}$ given in \eqref{eq:Renyi_strategy}. By construction of the local update rules in \eqref{eq:one_site_update}, each ``global'' probability measure $\pi_t$ is invariant with respect to the Markov transition kernel $\Prob_t$ given in \eqref{eq:Glauber}. Moreover, as the relative entropy bound \eqref{eq:KL_stepsize} shows, the probability measures $\pi_t$ and $\pi_{t+1}$ are close for every $t$. Finally, we will show that, under the condition $\Delta \beta < 1$, the conditional distributions $\Prob_t(\cdot|x)$ and $\Prob_t(\cdot|y)$ will be close whenever the action profiles $x$ and $y$ are close. As we will demonstrate shortly, these three properties together ensure that, at each time step $t$, the decentralized action profile distribution $\mu_t = \Prob_t \Prob_{t-1} \ldots \Prob_1 \mu_0$ will be close to its centralized counterpart $\pi_t$. On a ``big picture'' level, this argument is similar in spirit to the one used by Narayanan and Rakhlin \cite{Narayanan_Rakhlin_RW} to construct and analyze efficient algorithms for centralized online minimization of a sequence of linear functions on a compact convex subset of a finite-dimensional Euclidean space. However, here we are interested in {\em decentralized} algorithms for {\em discrete} optimization. Moreover, the overall proof in \cite{Narayanan_Rakhlin_RW} is rather technical, drawing on ideas from the Riemannian geometry of interior-point optimization algorithms \cite{Nesterov_Todd_Riemannian_IP} and random walks on convex bodies \cite{Lovasz_Simonovits}. By contrast, our proof is much simpler, and relies on the  notion of {\em positive Ricci curvature} of a Markov chain recently introduced by Ollivier \cite{Ollivier_Ricci} (the reader is invited to consult a recent paper by Joulin and Ollivier \cite{Joulin_Ollivier_MCMC} for examples of how Ricci curvature ideas can be used to get sharp estimates of convergence rates of MCMC algorithms).

To separate out the key ideas underlying our proof, we have split this section into three parts. The first part (Section~\ref{ssec:positive_Ricci}) uses the notion of Ricci curvature of Markov chains to obtain uniform error bounds for sampling from a time-varying sequence of probability measures. Because the results of this part may be of independent interest, we formulate them in a much more general setting of complete separable metric spaces. The second part (Section~\ref{ssec:Glauber_Ricci}) applies these results to the time-varying Glauber dynamics \eqref{eq:one_site_update}--\eqref{eq:Glauber}. Once all the necessary ingredients are in place, we complete the proof of Theorem~\ref{thm:LI_regret_bound} in the last part (Section~\ref{ssec:complete_proof}).

\subsubsection{Positive Ricci curvature and sampling from a time-varying sequence of probability measures}
\label{ssec:positive_Ricci}

Let $(\sX,\rho)$ be a complete separable metric space (i.e., a Polish space) equipped with the $\sigma$-algebra $\cB(\sX)$ of its Borel subsets. A {\em Markov transition kernel} on $\sX$ is a mapping $\Prob(\cdot|\cdot) : \cB(\sX) \times \sX \to [0,1]$, such that (i) $\Prob(\cdot|x)$ is a probability measure on $\sX$ for all $x$ and (ii) the mapping $x \mapsto \Prob(A|x)$ is measurable for every $A \in \cB(\sX)$. We define the action of a Markov kernel $\Prob$ on a probability measure $\mu \in \cP(\sX)$ as
\begin{align*}
	\Prob\mu(A) \deq \int_\sX \Prob(A|x)\mu(\d x),
\end{align*}
and we say that $\mu$ is {\em $\Prob$-invariant} if $\mu = \Prob\mu$. The {\em $L^1$ Wasserstein distance} (or {\em transportation distance}) \cite{Villani_topics} between probability measures $\mu,\nu \in \cP(\sX)$ is defined as
\begin{align}
	W_1(\mu,\nu) &\deq \inf_{\upsilon \in C(\mu,\nu)} \int_{\sX \times \sX} \rho(x,y) \upsilon(\d x, \d y), \label{eq:Wasserstein}
\end{align}
where $C(\mu,\nu)$ denotes the collection of all {\em couplings} of $\mu$ and $\nu$, i.e., all probability measures $\upsilon$ on $\sX \times \sX$ with marginals $\mu$ and $\nu$. An important {\em Kantorovich--Rubinstein theorem} (see, e.g., \cite[Theorem~1.14]{Villani_topics}) gives a variational representation of $W_1(\mu,\nu)$:
\begin{align}\label{eq:KR}
	W_1(\mu,\nu) = \sup_{f : \| f \|_{\rm Lip} \le 1} \left| \int_\sX f \d\mu - \int_\sX f \d\nu \right|,
\end{align}
where the supremum is over all real-valued functions $f$ on $\sX$ with Lipschitz constant
\begin{align*}
	\|f\|_{\rm Lip} \deq \sup_{x \neq y} \frac{|f(x)-f(y)|}{\rho(x,y)} \le 1.
\end{align*}
\begin{remark}\label{rem:TV_Wasserstein} {\em When $\rho$ is the {\em trivial metric}, i.e., $\rho(x,y) = \1_{\{x \neq y\}}$, the Wasserstein distance is equal to the total variation distance: for any $\mu,\nu \in \cP(\sX)$:
	\begin{align}\label{eq:TV_coupling}
		\| \mu - \nu \|_{\rm TV} = \inf_{\upsilon \in C(\mu,\nu)} \int_{\sX \times \sX}\1_{\{x \neq y\}} \upsilon(\d x, \d y)
	\end{align}
Moreover, for any two $\mu,\nu \in \cP(\sX)$, we can construct the so-called {\em optimal coupling} $\upsilon^\star \in C(\mu,\nu)$ that achieves the infimum in \eqref{eq:TV_coupling} (see, e.g., \cite[Section~4.2]{Levin_Peres_Wilmer}).}
\end{remark}

Fix a Markov kernel $\Prob$ on $\sX$. Following Ollivier \cite{Ollivier_Ricci}, we say that $\Prob$ has {\em positive Ricci curvature} if there exists some $\kappa \in (0,1]$, such that
\begin{align}\label{eq:Ricci_curvature}
	W_1\big(\Prob(\cdot|x),\Prob(\cdot|y)\big) \le (1-\kappa)\rho(x,y), \qquad \forall x,y \in \sX.
\end{align}
We will denote the supremum of all such $\kappa$ by $\Ric(\Prob)$ and call this number the Ricci curvature of $\Prob$. The following contraction inequality \cite[Proposition~20]{Ollivier_Ricci} is key: \eqref{eq:Ricci_curvature} holds for $\Prob$ with some $\kappa \in (0,1]$ if and only if
\begin{align}\label{eq:Ricci_contraction}
	W_1(\Prob\mu, \Prob\nu) \le (1-\kappa) W_1(\mu,\nu), \qquad \forall \mu,\nu \in \cP(\sX).
\end{align}
We are now ready to develop our main technical tool:

\anr{
\begin{lemma}\label{lm:Ricci_steps} Let $\Prob_1,\Prob_2,\ldots$ be a sequence of Markov kernels on $\sX$ with the following properties:
	\begin{itemize} 
		\item[(i)] 
		Each $\Prob_t$ has a unique invariant distribution $\pi_t$ and there exists some $\delta_t \in [0,1)$, such that
		\begin{align}\label{eq:small_steps}
			W_1(\pi_t,\pi_{t+1}) \le \delta_t, \qquad t = 1,\ldots,T.
		\end{align}
		\item[(ii)] 
		The Ricci curvatures of the $\Prob_t$'s are uniformly bounded from below 
		by some $\kappa^\star\in(0,1) $:
	\begin{align*}
		\Ric(\Prob_t) \ge \kappa^\star, \qquad t = 1,2,\ldots.
	\end{align*}
\end{itemize}
\an{Given a probability measure $\mu_1 \in \cP(\sX)$, let $\{\mu_t\}$ be a sequence of probability measures defined recursively via $\mu_{t+1} = \mu_t \Prob_t$.} Then, we have
\begin{align}\label{eq:Ricci_error_bound}
	W_1(\mu_t,\pi_t) \le (1-\kappa^\star)^{t-1} W_1(\mu_1,\pi_1) 
	+ \1_{\{t\ge 2\}}\sum_{s=1}^{t-1} (1-\kappa^\star)^{t-1-s}\delta_{s}, \qquad t \ge 1.
\end{align}
\end{lemma}
}
\anr{
\begin{proof} By inspection we can see that relation~\eqref{eq:Ricci_error_bound} holds for $t=1$. 
Next, note that for any $t\ge0$ we have
	\begin{align}
		W_1(\mu_{t+1},\pi_{t+1}) &\le W_1(\mu_{t+1},\pi_t) + W_1(\pi_t,\pi_{t+1}) \label{eq:Ricci_step1} \\
		&= W_1(\Prob_t \mu_t , \Prob_t  \pi_t) + W_1(\pi_t,\pi_{t+1}) \label{eq:Ricci_step2}\\
		&\le (1-\kappa^\star) W_1(\mu_t,\pi_t) + \delta_t, \label{eq:Ricci_step3}
	\end{align}
	where \eqref{eq:Ricci_step1} is by the triangle inequality, 
	\eqref{eq:Ricci_step2} uses the recursive definition of the $\mu_t$'s and the $\Prob_t$-invariance of $\pi_t$, and \eqref{eq:Ricci_step3} uses the contraction inequality \eqref{eq:Ricci_contraction} and the assumption \eqref{eq:small_steps}. 
	Thus, by induction on $t$, we can find that for all $t\ge1$,
	\begin{align*}
	W_1(\mu_{t+1},\pi_{t+1}) &\le (1-\kappa^\star)^t W_1(\mu_1,\pi_1) 
	+\sum_{s=1}^t (1-\kappa^\star)^{t-s}\delta_{s},
	\end{align*}
	which shows~\eqref{eq:Ricci_error_bound} for $t\ge 2$. 
\end{proof}
}
\anr{
\begin{corollary}\label{cor:Lipschitz_averages} 
Under the assumptions of Lemma~\ref{lm:Ricci_steps}, for any Lipschitz function $f : \sX \to \Reals$ we have
	\begin{align*}
		\left| \int_\sX f \d\mu_t - \int_\sX f \d\pi_t \right| 
		\le \| f \|_{\rm Lip} \left((1-\kappa^\star)^{t-1} W_1(\mu_1,\pi_1) 
	+ \1_{\{t\ge 2\}}\sum_{s=1}^{t-1} (1-\kappa^\star)^{t-1-s}\delta_{s}\right), \qquad t= 1,2\ldots.
	\end{align*}
\end{corollary}
\begin{proof} Use \eqref{eq:Ricci_error_bound} and the Kantorovich--Rubinstein formula \eqref{eq:KR}.
\end{proof}
}

\mr{
\begin{remark}{\em In the special case when $\delta_t = \delta$ for all $t$, the bounds of Lemma~\ref{lm:Ricci_steps} and Corollary~\ref{cor:Lipschitz_averages} become
	\begin{align*}
		W_1(\mu_t,\pi_t) \le \frac{\delta}{1-\kappa^\star}
	\end{align*}
	and
	\begin{align*}
	\left| \int_\sX f \d\mu_t - \int_\sX f \d\pi_t \right| \le \frac{\| f \|_{\rm Lip} \delta}{1-\kappa^\star}
	\end{align*}
respectively.}\end{remark}}

\subsubsection{Positive Ricci curvature of the time-varying Glauber dynamics}
\label{ssec:Glauber_Ricci}

We now particularize these results to our setting, where $\sX$ is the space of all tuples $x = \left( x_v : v \in V\right)$ equipped with the {\em Hamming distance}
\begin{align*}
	\rho_{\rm H}(x,y) \deq \sum_{v \in V}\1_{\{x_v \neq y_v\}}.
\end{align*}
In this case, the Ricci curvature bounds are equivalent to the so-called {\em path coupling} bounds of Bubley and Dyer \cite{Bubley_Dyer_path_coupling} (see also \cite[Chapter~14]{Levin_Peres_Wilmer} and \cite[Example~17]{Ollivier_Ricci}). In particular, in order to obtain a lower bound on the Ricci curvature of a given Markov kernel $\Prob$, it suffices to consider only those $x, y \in \sX$ with $\rho_{\rm H}(x,y) = 1$. Indeed, suppose that we can find some $\kappa \in (0,1]$, such that
\begin{align}\label{eq:Ricci_one}
	W_1\big(\Prob(\cdot|x),\Prob(\cdot|y)\big) \le 1-\kappa
	\qquad\hbox{for all $x,y$ with $\rho_{\rm H}(x,y)=1$. }
\end{align}
Then $\Ric(\Prob) \ge \kappa$. To see this, consider any pair $x, y \in \sX$ with $\rho_{\rm H}\an{(x,y)} = k$. 
Then, there exists a sequence $x_1,\ldots,x_{k+1} \in \sX$, such that $x_1 = x$, $x_{k+1} = y$, and $\rho_{\rm H}(x_j,x_{j+1})=1$ for all $1 \le j \le k$. Using this fact, we can write
\begin{align*}
	W_1\big(\Prob(\cdot|x),\Prob(\cdot|y)\big) &= W_1\big(\Prob(\cdot|x_1),\Prob(\cdot|x_{k+1})\big) \\
	&\le \sum^{k}_{j=1} W_1\big( \Prob(\cdot|x_j), \Prob(\cdot|x_{j+1})\big) \\
	&\le (1-\kappa)k \\
	&= (1-\kappa)\rho_{\rm H}(x,y),
\end{align*}
where the second step follows from the triangle inequality and the third step follows from \eqref{eq:Ricci_one}. Using this observation, we can prove the following:

\anr{
\begin{lemma}\label{lm:Ricci_bound} 
Let $\Prob_1,\ldots,\Prob_{T+1}$ be the Markov kernels on $\sX$ given by \eqref{eq:Glauber}, and 
let $\pi_1,\ldots,\pi_{T+1} \in \cP(\sX)$ be the Gibbs measures defined in \eqref{eq:pi_t}. 
Suppose that the parameter $\beta>0$ satisfies $\beta \Delta < 1$. Then, the conditions of Lemma~\ref{lm:Ricci_steps} are satisfied with
	\begin{align}\label{eq:W_stepsize}
		\delta_t = \frac{\beta|V|^2(\Delta+1)}{t+1}
	\end{align}
	and
	\begin{align}\label{eq:Glauber_Ricci_bound}
		\kappa^\star = \frac{1-\Delta \beta}{|V|}.
	\end{align}
\end{lemma}
}
\begin{proof}
The fact that each $\pi_t$ is invariant with respect to $\Prob_t$ follows from the detailed balance property \eqref{eq:reversibility}. To keep the paper relatively self-contained, we give in Appendix~\ref{app:Gibbs_sampler} a short proof of \eqref{eq:reversibility} as a consequence of a more general result on the Gibbs sampler.

	To upper-bound the Wasserstein distance $W_1(\pi_t,\pi_{t+1})$, we write
	\begin{align}
		W_1(\pi_t,\pi_{t+1}) &= \inf_{\upsilon \in C(\pi_t,\pi_{t+1})} \int_{\sX \times \sX} \rho_{\rm H}(x,y) \upsilon(\d x, \d y) \nonumber \\
		&\le |V| \int_{\sX \times \sX} \1_{\{x \neq y\}} \upsilon(\d x, \d y) \nonumber \\
		&= |V| \cdot \| \pi_t - \pi_{t+1} \|_{\rm TV}, \label{eq:W_stepsize_0}
	\end{align}
where in the first line we have used the definition \eqref{eq:Wasserstein} of the Wasserstein distance, 
while the last step follows from the coupling representation \eqref{eq:TV_coupling} of the total variation distance. 
\anr{Furthermore, using the CKKP inequality \eqref{eq:CKKP} 
and the relative-entropy bound \eqref{eq:KL_stepsize}, we get
	\begin{align*}
		\left\| \pi_t - \pi_{t+1} \right \|_{\rm TV} \le \sqrt{\frac{1}{2} D(\pi_t\|\pi_{t+1})}
		\le \frac{\beta|V|(\Delta+1)}{t+1}
		  \qquad\hbox{ for all $t\ge1$}.
	\end{align*}}
	Using this bound in \eqref{eq:W_stepsize_0}, we get \eqref{eq:W_stepsize}.

	Finally, we obtain a uniform lower bound on the Ricci curvature of the $\Prob_t$'s. Each $\Prob_t$ is of the form
	\begin{align*}
		\Prob_t(y|x) = \frac{1}{|V|}\sum_{v \in V}\Prob_{v,t}(y_v|x_{\partial v}) \1_{\{y_{-v} = x_{-v}\}},
	\end{align*}
	\anr{
	where
	\begin{align}\label{eq:local_Gibbs}
		\Prob_{v,t}(y_v|x_{\partial v}) 
		= \frac{\mu_{v,0}(y_v)\exp\left( -\beta F_{v,t-1}(y_v,x_{\partial v})\right)}{Z_{v,t}(x_{\partial v})}.
	\end{align}
	}
Recalling the discussion preceding the statement of the lemma, we only need to consider pairs $x,y$ with $\rho_{\rm H}(x,y)=1$. Fix such a pair $x,y$, and let $u \in V$ denote the single vertex at which they differ. We will construct a suitable coupling of $\Prob_t(\cdot|x)$ and $\Prob_t(\cdot|y)$. We define a random couple $(\bar{X},\bar{Y}) \in \sX \times \sX$ as follows. Select a vertex $v \in V$ uniformly at random. There are three cases to consider:
	\begin{itemize}
		\item If $v = u$, then $x_{-v} = x_{-u} = y_{-u} = y_{-v}$ and a fortiori
		\begin{align*}
			\Prob_{v,t}(\cdot|x_{\partial v}) = \Prob_{v,t}(\cdot|y_{\partial v}).
		\end{align*}
		In this case, we draw a random sample $A$ from $\Prob_{u,t}(\cdot|x_{\partial u})$ and let $\bar{X}_u = \bar{Y}_u = A$, $\bar{X}_{-u} = x_{-u}$, $\bar{Y}_{-u} = y_{-u}$. Then $\rho_{\rm H}(\bar{X},\bar{Y}) = 0$.
		\item If $v \in \partial u$, then we sample $(\bar{X}_v,\bar{Y}_v)$ from the optimal coupling of $\Prob_{v,t}(\cdot|x_{\partial v})$ and $\Prob_{v,t}(\cdot|y_{\partial v})$ (cf.~Remark~\ref{rem:TV_Wasserstein} in Section~\ref{ssec:positive_Ricci}), and let $\bar{X}_{-v} = x_{-v}$, $\bar{Y}_{-v} = y_{-v}$. 
		Then, \an{we have} $\bar{X}_v = \bar{Y}_v$ with probability $1-\left\| \Prob_{v,t}(\cdot|x_{\partial v})- \Prob_{v,t}(\cdot|y_{\partial v})\right\|_{\rm TV}$, in which case $\rho_{\rm H}(\bar{X},\bar{Y}) = \rho_{\rm H}(x,y)$; on the complementary event $\{\bar{X}_v \neq \bar{Y}_v\}$, the Hamming distance $\rho_{\rm H}(\bar{X},\bar{Y})$ will increase to $2$.
		\item If $v \not\in \partial_+u$, then $x_{\partial v} = y_{\partial v}$. We sample a random $A$ from $\Prob_{v,t}(\cdot|x_{\partial v}) = \Prob_{v,t}(\cdot|y_{\partial v})$ and let $\bar{X}_v = \bar{Y}_v = A$, $\bar{X}_{-v} = x_{-v}$, and $\bar{Y}_{-v} = y_{-v}$. In this case, $\rho_{\rm H}(\bar{X},\bar{Y}) = \rho_{\rm H}(x,y) = 1$.
	\end{itemize}
	Let $\bar{\upsilon}$ denote the joint probability distribution of $(\bar{X},\bar{Y})$. It is easy to show that $\bar{X}$ (respectively, $\bar{Y}$) has distribution $\Prob_t(\cdot|x)$ (respectively, $\Prob_t(\cdot|y)$). Therefore, $\bar{\upsilon}$ is an element of $C\big(\Prob_t(\cdot|x),\Prob_t(\cdot|y)\big)$. Moreover,
	\begin{align*}
	\int_{\sX \times \sX}\rho_{\rm H}\d\bar{\upsilon} &= 0 \cdot {\rm Pr}(v=u)  + 1 \cdot {\rm Pr}(v \not\in \partial_+ u) + \sum_{v' \in \partial u} \left(1+\left\| \Prob_{v',t}(\cdot|x_{\partial v'}) - \Prob_{v',t}(\cdot|y_{\partial v'}) \right\|_{\rm TV}\right) {\rm \Pr}(v=v') \\
		&\le 1 - \frac{\Delta+1}{|V|} + \frac{\Delta (1+\eta)}{|V|} \\
		&= 1 - \frac{1-\Delta\eta}{|V|},
	\end{align*}
	where
	\begin{align*}
		\eta = \max_{v \in \partial u}\left\| \Prob_{v,t}(\cdot|x_{\partial v}) - \Prob_{v,t}(\cdot|y_{\partial v}) \right\|_{\rm TV}.
	\end{align*}
	It remains to bound $\eta$ from above. To that end, we note that, for each $v \in V$, 
	both $\Prob_{v,t}(\cdot|x_{\partial v})$ and $\Prob_{v,t}(\cdot|y_{\partial v})$ are Gibbs measures, 
	cf.~\eqref{eq:local_Gibbs}. Therefore, using Lemma~\ref{lm:Gibbs}, we can write
	\anr{
	\begin{align}\label{eq:local_TV_bound}
		\left\| \Prob_{v,t}(\cdot|x_{\partial v}) - \Prob_{v,t}(\cdot|y_{\partial v}) \right\|_{\rm TV} 
		&\le \frac{\beta \left\| F_{v,t-1}(\cdot,x_{\partial v}) - F_{v,t-1}(\cdot,y_{\partial v}) \right\|_{\rm s}}{4}.
	\end{align}
	Using Lemma~\ref{lm:sup_norm_bound} in Appendix~\ref{app:misc} 
	and an argument similar to the one used to derive \eqref{eq:discounted_span_norm_bound}, we get
	\begin{align*}
		\left\| F_{v,t-1}(\cdot,x_{\partial v}) - F_{v,t-1}(\cdot,y_{\partial v}) \right\|_{\rm s} 
		&\le 2 \left\| F_{v,t-1}(\cdot,x_{\partial v}) - F_{v,t-1}(\cdot,y_{\partial v}) \right\|_\infty\cr
		&\le \frac{2}{t}\sum_{s=1}^{t-1}\|f_{v,s}(\cdot,x_{\partial v}) - f_{v,s}(\cdot,y_{\partial v})\|_\infty\cr
		& \le 4\rho_{\rm H}(x_{\partial v},y_{\partial v}).
	\end{align*}
Since $x$ and $y$ differ only at a single vertex, we have that $\rho_{\rm H}(x_{\partial v},y_{\partial v}) \le 1$. Therefore,
$$
\left\| F_{v,t-1}(\cdot,x_{\partial v}) - F_{v,t-1}(\cdot,y_{\partial v}) \right\|_{\rm s} \le 4.
$$
Note that this bound is \textit{independent} of $t$; this is a consequence of the $\frac{1}{t+1}$ scaling in Eq.~\eqref{eq:discount-fun}, which is in turn a direct consequence of the relative-entropy term in the instantaneous losses. Substituting this into \eqref{eq:local_TV_bound}, we get 
	\[\left\| \Prob_{v,t}(\cdot|x_{\partial v}) - \Prob_{v,t}(\cdot|y_{\partial v}) \right\|_{\rm TV} 
		\le \beta.\]
 Therefore, by the definition of $W_1$ it follows that
	\begin{align*}
		W_1\big( \Prob_t(\cdot|x), \Prob_t(\cdot|y)\big) \le \int_{\sX \times \sX} \rho_{\rm H}\d\bar{\upsilon} \le 1-\frac{1-\Delta \beta}{|V|},
	\end{align*}
	which in view of relation~\eqref{eq:Ricci_one}  yields \eqref{eq:Glauber_Ricci_bound}.
	}
\end{proof}

\subsubsection{Completing the proof}\label{ssec:complete_proof}
We decompose the regret $R^{\rm LI}_T(f^T)$ as follows:
\begin{align}\label{eq:LI_regret_decomposition}
	R^{\rm LI}_T(f^T) &= \sum^T_{t=1} \ell_t(\mu_t) - \inf_{\nu \in \cP(\sX)} \sum^T_{t=1}\ell_t(\nu) \nonumber\\
	&= \sum^T_{t=1}\left( \ell_t(\mu_t) - \ell_t(\pi_t)\right) + \sum^T_{t=1} \ell_t(\pi_t) - \inf_{\nu \in \cP(\sX)} \sum^T_{t=1}\ell_t(\nu) \nonumber \\
	&\le \sum^T_{t=1}\left( \ell_t(\mu_t) - \ell_t(\pi_t)\right) + R_T(f^T).
\end{align}
Next, we use the form of the instantaneous costs $\ell_t$ to expand the first summation on the right-hand side of \eqref{eq:LI_regret_decomposition}:
\begin{align}\label{eq:generic}
	\sum^T_{t=1}\left( \ell_t(\mu_t) - \ell_t(\pi_t)\right) 
	&= \sum^T_{t=1} \beta \left( \ave{\mu_t,f_t} - \ave{\pi_t,f_t}\right) 
	+ \sum^T_{t=1} \left(D(\mu_t \| \mu_0) - D(\pi_t \| \mu_0)\right).
\end{align}
By Lemma~\ref{lm:Lipschitz_bound}, each $f_t$ is Lipschitz with respect to the Hamming metric with constant 
$2 |V|(\Delta+1)$. 
\anr{Therefore, using Corollary~\ref{cor:Lipschitz_averages}, we obtain
\begin{align*}
	\ave{\mu_t,f_t} - \ave{\pi_t,f_t} &\le \| f \|_{\rm Lip} \left((1-\kappa^\star)^{t-1} W_1(\mu_1,\pi_1) 
	+ \1_{\{t\ge 2\}}\sum_{s=1}^{t-1} (1-\kappa^\star)^{t-1-s}\delta_{s}\right)\cr
	&\le 2 |V|(\Delta+1)\left((1-\kappa^\star)^{t-1} W_1(\mu_1,\pi_1) 
	+ \1_{\{t\ge 2\}}\sum_{s=1}^{t-1} (1-\kappa^\star)^{t-1-s}\delta_{s}\right).
\end{align*}
Using the expression for $\delta_t$ given in Lemma~\ref{lm:Ricci_bound}, we further obtain
\begin{align*}
	\ave{\mu_t,f_t} - \ave{\pi_t,f_t} &\le 
	 2 |V|(\Delta+1)\left((1-\kappa^\star)^{t-1} W_1(\mu_1,\pi_1) 
	+ \1_{\{t\ge 2\}}\sum_{s=1}^{t-1} (1-\kappa^\star)^{t-1-s}\beta|V|^2(\Delta+1)\gamma_{s+1}\right).
\end{align*}
Therefore,
 \begin{align}
  \sum^T_{t=1} \beta \left( \ave{\mu_t,f_t} - \ave{\pi_t,f_t}\right) 
  & \le  2 \beta|V|(\Delta+1)W_1(\mu_1,\pi_1) \sum^T_{t=1}(1-\kappa^\star)^{t-1} 
  + 2\beta^2|V|^3(\Delta+1)^2\sum^T_{t=2}\sum_{s=1}^{t-1} (1-\kappa^\star)^{t-1-s}\gamma_{s+1}\qquad\cr
  & \le 2 \beta|V|(\Delta+1)W_1(\mu_1,\pi_1) \frac{1}{\kappa^\star} 
  + 2\beta^2|V|^3(\Delta+1)^2\sum^{T-1}_{\tau=1}\sum_{s=1}^{\tau} (1-\kappa^\star)^{\tau-s}\gamma_{s+1},
  \label{eq:sec-row}
 \end{align}
  where the last inequality is obtained by using 
  \begin{align}\label{eq:easy}
  \sum^T_{t=1}(1-\kappa^\star)^{t-1} \le \sum^\infty_{t=1}(1-\kappa^\star)^{t-1} =\frac{1}{k^\star},\end{align}
  and by letting $\tau=t-1$ in the second sum over $t$. 
  By exchanging the order of summation in the last term in~\eqref{eq:sec-row} and using~\eqref{eq:easy}, 
  we have
 \[\sum^{T-1}_{\tau=1}\sum_{s=1}^{\tau} (1-\kappa^\star)^{\tau-s}\gamma_{s+1} 
 =  \sum_{t=2}^{T} \gamma_t\sum^{T-t}_{\tau=0}(1-\kappa^\star)^{\tau}
 \le \frac{1}{\kappa^\star}\sum_{t=2}^{T} \gamma_t,\]
 implying that 
 \begin{align*}
  \sum^T_{t=1} \beta \left( \ave{\mu_t,f_t} - \ave{\pi_t,f_t}\right) 
  \le 2 \beta|V|(\Delta+1)W_1(\mu_1,\pi_1) \frac{1}{\kappa^\star} 
  + 2\beta^2|V|^3(\Delta+1)^2\frac{1}{\kappa^\star} \sum_{t=2}^{T} \gamma_t.
  \end{align*}
  Since $\gamma_t=\frac{1}{t}$ for all $t \ge 1$, it follows that 
  \begin{align}\label{eq:f-term}
  \sum^T_{t=1} \beta \left( \ave{\mu_t,f_t} - \ave{\pi_t,f_t}\right) 
   & \le 2 \beta|V|(\Delta+1)W_1(\mu_1,\pi_1) \frac{1}{\kappa^\star} 
  + 2\beta^2|V|^3(\Delta+1)^2\frac{1}{\kappa^\star} \sum_{t=2}^{T} \frac{1}{t}\cr
  & \le 2 \beta|V|(\Delta+1)W_1(\mu_1,\pi_1) \frac{1}{\kappa^\star} 
  + 2\beta^2|V|^3(\Delta+1)^2\frac{1}{\kappa^\star} \ln T,
  \end{align}
  where the last inequality follows from
  \[\sum_{t=2}^{T} \frac{1}{t}\le \int_{1}^T\frac{\d t}{t} = \ln T.\]
}

\anr{
Next, we deal with the relative entropy difference term in~\eqref{eq:generic}. Given a probability distribution $\mu \in \cP(\sX)$, let $H(\mu) = - \ave{\mu, \ln \mu}$ denote its 
{\em Shannon entropy} \cite{CovTho06}. Then
\begin{align}
	D(\mu_t \| \mu_0) - D(\pi_t \| \mu_0) &= H(\pi_t) - H(\mu_t) + \Ave{\mu_t,\ln \frac{1}{\mu_0}} - \Ave{\pi_t, \ln \frac{1}{\mu_0}} \nonumber\\
	&\le |H(\pi_t)-H(\mu_t)| +  \| \pi_t - \mu_t \|_{\rm TV} \cdot |V| \ln \frac{1}{\theta_d}, \label{eq:KL_difference}
\end{align}
where $\theta_d = \min_{v \in V}\min_{a \in \{1,\ldots,q\}} \mu_{v,0}(a)$.  To upper-bound the first term in \eqref{eq:KL_difference}, we use the following continuity estimate for the Shannon entropy (see, e.g.,~\cite[Theorem~17.3.3]{CovTho06}): 
For any two $\mu,\nu \in \cP(\sX)$ with $\| \mu - \nu \|_{\rm TV} \le 1/4$,
\begin{align*}
	|H(\mu) - H(\nu)| \le 2\left( \| \mu - \nu \|_{\rm TV} \ln |\sX| +  \| \mu - \nu \|_{\rm TV} \ln \frac{1}{\| \mu - \nu \|_{\rm TV}}\right),
\end{align*}
where $|\sX| = q^{|V|}$ is the cardinality of $\sX$. In order to use this estimate, we need an upper bound on 
$\| \pi_t - \mu_t \|_{\rm TV}$, which can be obtained as follows:
\begin{align}
	\| \pi_t - \mu_t \|_{\rm TV} 
	&= \inf_{\upsilon \in C(\pi_t,\mu_t)} \int_{\sX \times \sX} \1_{\{x \neq y\}} \upsilon(\d x, \d y) \nonumber\\
	&\le \inf_{\upsilon \in C(\pi_t,\mu_t)} \int_{\sX \times \sX} \rho_{\rm H}(x,y) \upsilon(\d x, \d y) \nonumber \\
	&= W_1(\pi_t,\mu_t) \nonumber\\
	&\le (1-\kappa^\star)^{t-1} W_1(\mu_1,\pi_1) 
	+ \1_{\{t\ge 2\}}\sum_{s=1}^{t-1} (1-\kappa^\star)^{t-1-s}\delta_{s},\label{eq:TV_stepsize_bound}
\end{align}
where the last step follows from Lemma~\ref{lm:Ricci_steps} (see~\eqref{eq:Ricci_error_bound}).} \mr{We can upper-bound the Wasserstein distance $W_1(\mu_1,\pi_1)$ as follows:
\begin{align*}
	W_1(\mu_1,\pi_1) &= \inf_{\upsilon \in C(\mu_1,\pi_1)} \int_\sX \rho_{\rm H}(x,y) \upsilon(\d x, \d y) \\
	&\le |V| \inf_{\upsilon \in C(\mu_1,\pi_1)} \int_\sX \1_{\{x\neq y\}} \upsilon(\d x, \d y) \\
	&= |V| \| \mu_1 - \pi_1 \|_{\rm TV} \\
	&\le |V|.
\end{align*}
Using this and the fact that $\delta_s \le \frac{\beta |V|^2 (\Delta+1)}{s+1}$ by Lemma~\ref{lm:Ricci_bound} in \eqref{eq:TV_stepsize_bound}, we can write
\begin{align}
	\| \pi_t - \mu_t \|_{\rm TV} &\le K \sum^t_{s=1} \frac{(1-\kappa^\star)^{t-s}}{s} \nonumber\\
	&= K p_t\left(1-\kappa^\star\right),
\end{align}
where $K \equiv K(\beta,|V|,\Delta) \deq \max \left\{ |V|, \beta |V|^2 (\Delta+1) \right\}$, and $p_t(u) \deq \sum^t_{s=1} \frac{u^{t-s}}{s}$. As a consequence of Lemma~\ref{lm:recurrence} in Appendix~\ref{app:recursion}, there exists a finite $T_0 = T_0(\beta,|V|,\Delta+1)$, such that the sequence $\left\{ p_t\left(1-\kappa^\star\right)\right\}^\infty_{t=T_0}$ is strictly decreasing and convergent to zero. Therefore, there exists a finite $T_1 \equiv T_1(\beta,|V|,\Delta)$, such that
\begin{align*}
	\| \pi_t - \mu_t \|_{\rm TV} \le K p_t\left(1-\kappa^\star\right) \le \frac{1}{4}, \qquad t \ge T_1.
\end{align*}
Moreover, the function $u \mapsto -u \ln u$ is increasing on the open interval $(0,1/e)$, so for $t \ge T_1$ we have
\begin{align*}
	\| \pi_t - \mu_t \|_{\rm TV} \ln \frac{1}{\| \mu_t - \pi_t \|_{\rm TV}} &\le K p_t\left(1-\kappa^\star\right) \ln \frac{1}{K p_t\left(1-\kappa^\star\right)} \\
	&\le Kp_t\left(1-\kappa^\star\right) \ln t,
\end{align*}
where we have also used the fact that $p_t(u) \ge 1/t$.  Consequently,
\begin{align}
	D(\mu_t \| \pi_0) - D(\pi_t \| \pi_0) &\le \begin{cases}
	|V| \ln \frac{q^2}{\theta_d}, & t < T_1 \\
 	Kp_t\left(1-\kappa^\star\right)\left(|V| \ln \frac{q^2}{\theta_d} + \ln t\right),&  t \ge T_1
\end{cases}.
\end{align}
Summing from $t=1$ to $t=T$, we get
\begin{align}
	\sum^T_{t=1} \left[ D(\mu_t \| \pi_0) - D(\pi_t \| \pi_0) \right] 
	&\le T_1 |V| \ln \frac{q^2}{\theta_d} + K\left(|V| \ln \frac{q^2}{\theta_d} + \ln T\right) \sum^T_{t=1} p_t\left( 1-\kappa^\star\right) \nonumber \\
	&= T_1 |V| \ln \frac{q^2}{\theta_d} + K\left(|V| \ln \frac{q^2}{\theta_d} + \ln T\right)\sum^T_{t=1}\sum^t_{s=1} \frac{(1-\kappa^\star)^{t-s}}{s} \nonumber\\
	&\le T_1 |V| \ln \frac{q^2}{\theta_d} + K\left(|V| \ln \frac{q^2}{\theta_d} + \ln T\right)\frac{1}{\kappa^\star} \ln(T+1).\label{eq:d-term}
\end{align}
Combining \eqref{eq:f-term} and \eqref{eq:d-term}, we get
\begin{align*}
&	\sum^T_{t=1}\left[\ell_t(\mu_t) - \ell_t(\pi_t)\right] \\
&\quad\le \sum^T_{t=1}\left[\beta\left|\ave{\mu_t,f_t} - \ave{\pi_t,f_t}\right| + \left| D(\mu_t \| \mu_{0}) - D(\pi_t \| \mu_0) \right|\right] \\
&\quad\le \frac{2 \beta|V|^2(\Delta+1)}{\kappa^\star} + T_1 |V| \ln \frac{q^2}{\theta_d} 
	  +\frac{1}{\kappa^\star}\left( 2\beta^2|V|^3(\Delta+1)^2 +  K\left(|V| \ln \frac{q^2}{\theta_d} + \ln T\right)\right)\ln (T+1) 
\end{align*}
We can now obtain the bound on the overall regret, via \eqref{eq:LI_regret_decomposition}:
\begin{align}
	R^{\rm LI}_T(f^T) &\le \frac{1}{\kappa^\star}\left( 2\beta^2|V|^3(\Delta+1)^2 +  K\left(|V| \ln \frac{q^2}{\theta_d} + \ln T\right)\right)\ln (T+1) \nonumber\\
		& \qquad + 2\left(\beta|V|(\Delta+1) \right)^2\ln(T+1)	+ T_1 |V| \ln \frac{q^2}{\theta_d} + \frac{2 \beta|V|^2(\Delta+1)}{\kappa^\star}  + \ln \frac{1}{\theta}.
\end{align}
Substituting the expression for $\kappa^\star$ from Lemma~\ref{lm:Ricci_bound}, we get \eqref{eq:LI_regret_bound}, and the proof is complete.
}

\section{Conclusion}

We have studied a model of online (i.e., real-time) discrete optimization by a social network 
consisting of agents that must choose actions to balance \an{their} immediate time-varying costs against a tendency to act according to some default myopic strategy. The costs are generated by a dynamic environment, and the agents lack ability or incentive to construct an a priori model of the environment's evolution. 
The global cost of the network decomposes into a sum of individual and pairwise \an{local-}interaction terms and, at each time step, \an{every} agent is informed only about its own cost and the pairwise costs in its immediate neighborhood. These assumptions on the network and on the environment capture the so-called {\em Knightian uncertainty} \cite{Knight_book,Bewley_Knightian1,Bewley_Knightian2}. The overall objective is to minimize the worst-case regret, i.e., the difference between the cumulative real-time performance of the network and the best performance that could have been achieved in hindsight with full centralized knowledge. We have constructed an explicit strategy for the network based on the Glauber dynamics and showed that it achieves favorable scaling of the regret in terms of problem parameters under a Dobrushin-type mixing condition. Our proof uses ideas from statistical physics, as well as recent developments in the theory of Markov chains in metric spaces, specifically Ollivier's notion of positive Ricci curvature of a Markov operator \cite{Ollivier_Ricci}.

Although the notion of regret is backward-looking, it is important conceptually since it  quantifies the agents' ability to make {\em forecasts} even in the absence of a Bayesian model, and to improve their decisions over time.  From the point of view of economics, regret minimization is significant for two reasons.  First, from the positive (or descriptive) standpoint, it allows for boundedly rational agents. Second, it may be used as a basis for what Selten \cite{Selten} has called a {\em practically normative} theory of economic behavior, since the goal of minimizing regret is synonymous with using past experience to improve one's decisions in the future, as opposed to following a strategy based on ideal rational expectations independent of the environment. In addition, in the online learning framework, the model of the interaction between the social network and the environment does not rely on probability judgments or assumptions about {\em what will happen}. Rather, probability is used as a tool to help the agents decide {\em what to do} -- how to allocate priority to different actions? When to perform experimentation, and when to stick with a strategy that had performed well in the past? Thus, probability is used as an objective {\em evolutionary mechanism} for selecting an action \cite{NelsonWinter,Selten} or as a mechanism to keep track of past experience in a case-based decision framework \cite{GilboaSchmeidler_book}, rather than as a subjective {\em belief} about the environment. This viewpoint is, of course, ideally suited for a Knightian theory of decision-making, and it meshes well with post-Keynesian critiques of the use of probability to quantify uncertainty \cite{Davidson,Crotty}.

\begin{appendix}

	\renewcommand{\theequation}{\Alph{section}.\arabic{equation}}
	\renewcommand{\thelemma}{\Alph{section}.\arabic{lemma}}
	\setcounter{equation}{0}
	\setcounter{lemma}{0}

\section{Proof of Lemma~\ref{lm:Gibbs}}
\label{app:Gibbs_lemma}

All Gibbs measures $\mu_g$  induced by the same base measure $\mu$
have the same support as $\mu$. Therefore, the quantity $D(\mu_g \| \mu_h)$ is finite for all functions $g$ and $h$ on $\sX$, and
\begin{align}
	D(\mu_g \| \mu_h) &= \left\langle \mu_g, \ln \frac{\mu_g}{\mu_h}\right\rangle \nonumber \\
	&= \left\langle \mu_g, g-h\right\rangle + \ln \frac{\ave{\mu,\exp(h)}}{\ave{\mu,\exp(g)}} \nonumber \\
	&= \left\langle \mu_g, g-h\right\rangle + \ln \left\langle\mu_g, \exp(h-g)\right\rangle.\label{eq:KL_Gibbs_0}
\end{align}
We now use the well-known Hoeffding bound \cite{Hoe63}, which for our purposes can be stated as follows: for any function $F : \sX \to \Reals$ and any $\nu \in \cP(\sX)$,
\begin{align}\label{eq:Hoeffding}
	\ln \ave{\nu, \exp(F)} \le \ave{\nu, F} + \frac{\| F \|^2_{\rm s}}{8}.
\end{align}
Applying \eqref{eq:Hoeffding} to the second term in \eqref{eq:KL_Gibbs_0}, we note that the terms involving the expectation of $g-h$ with respect to $\mu_g$ cancel, and we are left with \eqref{eq:KL_Gibbs}. The bound \eqref{eq:TV_Gibbs} follows from \eqref{eq:KL_Gibbs} and the CKKP inequality \eqref{eq:CKKP}.

\setcounter{equation}{0}
\setcounter{lemma}{0}

\section{Gibbs sampler and detailed balance}
\label{app:Gibbs_sampler}

In order to keep the paper self-contained, we give a brief proof of the detailed balance property of the discrete-state Gibbs sampler \cite{Tierney_MCMC}. Consider an arbitrary everywhere positive probability measure $\pi \in \cP(\sX)$ and a random variable $X = (X_v)_{v \in V}$ with distribution $\pi$. For any $v \in V$,  the conditional probability that $X_v = x_v$ given $X_{-v} = x_{-v}$ is equal to
\begin{align*}
	\pi_v(x_v|x_{-v}) \deq \frac{\pi(x_v,x_{-v})}{\pi_{-v}(x_{-v})},
\end{align*}
where $(a,x_{-v})$ denotes the tuple $y \in \sX$ obtained from $x$ by replacing $x_v$ with $a$, i.e., $y_v = a$ and $y_{-v} = x_{-v}$, and
\begin{align*}
\pi_{-v}(x_{-v}) = \sum_{a \in \{1,\ldots,q\}} \pi(a,x_{-v}).
\end{align*}
The  {\em Gibbs sampler} is implemented as follows: starting from $x \in \sX$, pick a vertex $v \in V$ uniformly at random, replace $x_v$ with a random sample $Y_v$ from $\pi_v(\cdot|x_{-v})$, and let $Y_{-v} = x_{-v}$. The overall stochastic transformation $x \to Y$ is described by the Markov kernel
\begin{align*}
	\Prob(y|x) = \frac{1}{|V|}\sum_{v \in V} \pi_v(y_v|x_{-v})\1_{\{x_{-v}=y_{-v}\}}.
\end{align*}
Then we claim that the pair $(\pi,\Prob)$ has the detailed balance property
\begin{align*}
	\pi(x)\Prob(y|x) = \pi(y)\Prob(x|y), \qquad \forall x,y \in \sX.
\end{align*}
Indeed,
\begin{align*}
	\pi(x)\Prob(y|x) &= \frac{1}{|V|}\sum_{v \in V} \pi_v(y_v|x_{-v}) \pi(x_v,y_{-v}) \1_{\{x_{-v} = y_{-v}\}} \\
	&= \frac{1}{|V|}\sum_{v \in V} \frac{\pi(y_v,x_{-v})}{\pi_{-v}(x_{-v})} \pi(x_v,y_{-v}) \1_{\{x_{-v} = y_{-v}\}} \\
	&= \frac{1}{|V|}\sum_{v \in V} \frac{\pi(y_v,x_{-v})}{\pi_{-v}(y_{-v})} \pi(x_v,y_{-v}) \1_{\{x_{-v} = y_{-v}\}} \\
	&= \frac{1}{|V|}\sum_{v \in V} \frac{\pi(x_v,y_{-v})}{\pi_{-v}(y_{-v})}\pi(y_v,x_{-v})  \1_{\{x_{-v} = y_{-v}\}} \\
	&= \frac{1}{|V|}\sum_{v \in V} \pi_{-v}(x_v|y_{-v}) \pi(y_v,x_{-v}) \1_{\{x_{-v} = y_{-v}\}} \\
	&= \pi(y) \Prob(x|y).
\end{align*}
A simple calculation shows that when \an{$\pi=\mu_f$ for} a Gibbs measure $\mu_f$ 
\an{induced by} an everywhere positive product measure $\mu \in \cP(\sX)$ and any \an{function} 
$f \in \cF$, the conditional measure $\pi_v(\cdot|x_{-v})$ for any $v \in V$ has the form
\begin{align*}
	\pi_{-v}(\cdot|x_{-v}) \propto \mu_v(x_v)\exp\left(-f_v(\cdot,x_{\partial v})\right).
\end{align*}
This, in turn, implies the detailed balance property \eqref{eq:reversibility}.

\setcounter{equation}{0}
\setcounter{lemma}{0}

\section{A polynomial recurrence in the proof of Theorem~\ref{thm:LI_regret_bound}}
\label{app:recursion}

For each $t=1,2,\ldots$, consider the polynomial
\begin{align*}
	p_t(u) = \sum^t_{s=1} \frac{u^{t-s}}{s}.
\end{align*}
We are interested in its behavior on the interval $[0,1]$.
\begin{lemma}\label{lm:recurrence} For each $u \in [0,1)$, the sequence $\{p_t(u)\}^\infty_{t=1}$ converges to zero. Moreover, there exists a finite $t_0 = t_0(u) \in \Naturals$, such that $p_{t+1}(u) < p_t(u)$ for all $t \ge t_0$.
\end{lemma}
\begin{proof} We first observe the following recurrence relation: for any $u \in [0,1]$,
	\begin{align}\label{eq:p_recurrence}
		p_{t+1}(u) = up_t(u) + \frac{1}{t+1}.
	\end{align}
From this, we see that $\lim_{t \to \infty} p_t(u) = 0$ for any $u \in [0,1)$. Let us fix some such $u$. Suppose that $p_{t_0+1}(u) < p_{t_0}(u)$ for some $t_0$. Then we claim that $p_{s+1}(u) < p_s(u)$ for all $s \ge t_0$. Indeed, from \eqref{eq:p_recurrence},
\begin{align*}
	p_{t_0+2}(u) &= up_{t_0+1}(u) + \frac{1}{t_0+2} \\
	&< up_{t_0}(u) + \frac{1}{t_0+2} \\
	&< up_{t_0}(u) + \frac{1}{t_0+1} \\
	&= p_{t_0+1}(u).
\end{align*}
The general claim of strict monotonicity then follows by induction. It remains to prove that such a finite $t_0$ always exists. To that end, consider for arbitrary $t$ the polynomial
\begin{align*}
	q_t(u) \deq p_{t+1}(u) - p_t(u) = u^{t+1} - \sum^{t-1}_{s=1} \frac{u^{t+1-s}}{s(s+1)}.
\end{align*}
The leading coefficient of $q_t$ is positive while all other coefficients are negative, so, by Descartes' rule of signs, $q_t$ has exactly one positive real root. Let us denote this root by $u_t$. We claim that $u_t \in (0,1]$. Indeed, $u_t$ must be positive, since $q_t(u)$ has a nonzero constant term. Moreover, $q_t(0) = -\frac{1}{t(t+1)}$ and $q_t(1) = \frac{1}{t+1}$, so $q_t$ changes sign in $[0,1]$. Thus, $u_t \in (0,1]$, and
\begin{align*}
	p_{t+1}(u) < p_t(u), \qquad u < u_t.
\end{align*}
By virtue of this strict monotonicity property, the sequence $\{u_t\}^\infty_{t=1}$ is strictly increasing and bounded by one. Now, for a given $u$ simply take $t_0$ to be the smallest element of the set $\{ t \in \Naturals: u_t > u \}$. 
\end{proof}

\setcounter{equation}{0}
\setcounter{lemma}{0}

\section{Miscellanea}
\label{app:misc}

\begin{lemma}\label{lm:sup_norm_bound} Consider all functions $f : \sX \to \Reals$ of the form \eqref{eq:cost_function}, where all local terms $\phi_v$ and $\phi_{u,v}$ take values in the interval $[-1,1]$. Then
	\begin{align}\label{eq:sup_norm_bound}
		\| f \|_\infty \le |V|(\Delta+1),
	\end{align}
	where $\| f \|_\infty \deq \max_{x \in \sX}|f(x)|$ is the sup norm of $f$.
Moreover, for any $x,y \in \sX$ and any $v \in V$,
	\begin{align}\label{eq:local_Lipschitz_bound}
		\left\| f_{v}(\cdot,x_{\partial v}) - f_{v}(\cdot,y_{\partial v}) \right\|_\infty \le 2\rho_{\rm H}(x_{\partial v},y_{\partial v}),
	\end{align}
	where
	\begin{align*}
		\rho_{\rm H}(x_{\partial v},y_{\partial v}) = \sum_{u \in \partial v}\1_{\{x_u \neq y_u\}}.
	\end{align*}
\end{lemma}
\begin{proof} For any $x \in \sX$, we have
	\begin{align*}
		|f(x)| &\le \sum_{v \in V} \left| \phi_v(x_v) \right| + \sum_{\edge{u}{v} \in E}\left|\psi_{u,v}(x_u,x_v)\right| \\
		&\le  |V| + |E|.
	\end{align*}
Since the graph $G = (V,E)$ is undirected and simple, an elementary counting argument shows that 
\[\an{|E| \le |V|\Delta/2}.\] 
Overbounding slightly, we get \eqref{eq:sup_norm_bound}. Similarly, for any $a \in \{1,\ldots,q\}$,
\begin{align*}
	\left| f_{v}(a,x_{\partial v}) - f_{v}(a,y_{\partial v}) \right| &\le \sum_{u \in \partial v} \left| \psi_{u,v}(a,x_{\partial v}) - \psi_{u,v}(a,y_{\partial v})\right| \\
	&\le 2 \sum_{u \in \partial v} \1_{\{x_u \neq y_u\}} \\
	&= 2 \rho_{\rm H}(x_{\partial v}, y_{\partial v}),
\end{align*}
which gives us \eqref{eq:local_Lipschitz_bound}.
\end{proof}

\begin{lemma}\label{lm:Lipschitz_bound} Under the same assumptions as in Lemma~\ref{lm:sup_norm_bound}, each cost function $f$ of the form \eqref{eq:cost_function} is Lipschitz with respect to the Hamming distance $\rho_{\rm H}$, with Lipschitz constant $\| f \|_{\rm Lip} \le 2|V|(\Delta+1)$. 
\end{lemma}
\begin{proof} For any two $x,y \in \sX$, we have
	\begin{align*}
		\left|f(x) - f(y)\right| &\le \sum_{v \in V}\left|\phi_v(x_v) - \phi_v(y_v)\right| + \sum_{\edge{u}{v} \in E}\left|\psi_{u,v}(x_u,x_v) - \psi_{u,v}(y_u,y_v)\right| \\
		&\le 2\left\{\sum_{v \in V} \1_{\{x_v \neq y_v\}} +\sum_{\edge{u}{v} \in E} \1_{\{(x_u,x_v) \neq (y_u,y_v)\}}\right\} \\
		&\le 2\left\{\sum_{v \in V} \1_{\{x_v \neq y_v\}} + \sum_{\edge{u}{v} \in E} \left( \1_{\{x_u \neq y_u\}} + \1_{\{ x_v \neq y_v \}} \right)\right\} \\
		&\le 2(|V|+2|E|)\sum_{v \in V}\1_{\{x_v \neq y_v\}} \\
		&= 2|V|(\Delta+1) \rho_{\rm H}(x,y).
	\end{align*}
\end{proof}

\end{appendix}

\end{document}